\documentclass{daj}
\usepackage[utf8]{inputenc}
\usepackage{amsmath,amsthm,amssymb}
\usepackage{graphics}
\usepackage{hyperref}
\usepackage{tikz,tikz-cd}
\usepackage{mathtools}
\usepackage{newtxmath,newtxtext}
\usepackage[square,sort,comma,numbers]{natbib}
\definecolor{darkblue}{rgb}{0.0,0.0,0.3}

\dajAUTHORdetails{%
  title = {Improved bounds on number fields of small degree},
  author = {Theresa C. Anderson, Ayla Gafni, Kevin Hughes, Robert J. Lemke Oliver, David Lowry-Duda, Frank Thorne, Jiuya Wang, and Ruixiang Zhang},
  plaintextauthor = {Theresa C. Anderson, Ayla Gafni, Kevin Hughes, Robert J. Lemke Oliver, David Lowry-Duda, Frank Thorne, Jiuya Wang, Ruixiang Zhang},
  runningauthor = {Anderson, Gafni, Hughes, Lemke Oliver, Lowry-Duda, Thorne, Wang, and Zhang},
}   

\dajEDITORdetails{%
   year={2024},
   number={19},
   received={9 September 2022},   
   revised={26 October 2023},    
   published={23 December 2024},  
   doi={10.19086/da.127737},       
}   

\theoremstyle{plain}
\newtheorem{theorem}{Theorem}[section]
\newtheorem*{theorem*}{Theorem}
\newtheorem{lemma}[theorem]{Lemma}
\newtheorem{proposition}[theorem]{Proposition}
\newtheorem*{proposition*}{Proposition}
\newtheorem{corollary}[theorem]{Corollary}
\newtheorem*{corollary*}{Corollary}

\theoremstyle{definition}
\newtheorem{remark}[theorem]{Remark}
\newtheorem*{remark*}{Remark}

\numberwithin{equation}{section}

\newcommand{\Z}{\mathbb{Z}}

\DeclareMathOperator{\rad}{rad}
\DeclareMathOperator{\Disc}{Disc}
\DeclareMathOperator{\disc}{disc}
\DeclareMathOperator{\idx}{index} 

\def\dd{{\,{\rm d}}}

\begin{document}

\begin{frontmatter}[classification=text]

\title{Improved bounds on number fields of small degree} 

\author[Anderson]{Theresa C. Anderson\thanks{Supported by NSF DMS-2231990, DMS-1954407, and CAREER DMS-2237937}}
\author[Gafni]{Ayla Gafni\thanks{Supported by Simons Collaboration Grant No.~946710}}
\author[Hughes]{Kevin Hughes\thanks{Supported by the Additional Funding Programme for Mathematical Sciences, delivered by EPSRC (EP/V521917/1) and the Heilbronn Institute for Mathematical Research}}
\author[Lemke Oliver]{Robert J. Lemke Oliver\thanks{Supported by NSF DMS-2200760 and a Simons Fellowship in Mathematics}}
\author[Lowry-Duda]{David Lowry-Duda\thanks{Supported by the Simons Collaboration in Arithmetic Geometry, Number Theory, and Computation via the Simons Foundation Grant No.~546235}}
\author[Thorne]{Frank Thorne\thanks{Supported by NSF DMS-2101874 and a Simons Collaboration Grant No.~586594}}
\author[Wang]{Jiuya Wang\thanks{Supported by NSF DMS-2201346}}
\author[Zhang]{Ruixiang Zhang\thanks{Supported by NSF DMS-1856541 and NSF DMS-2207281}}

\begin{abstract}
We study the number of degree $n$ number fields with discriminant bounded by $X$.
In this article, we improve an upper bound due to Schmidt on the number of such fields that was previously the best known upper bound for $6
\leq n \leq 94$.
\end{abstract}
\end{frontmatter}

\section{Introduction and outline of paper}\label{sec:introduction}

Let $N_n(X) := \# \{ K / \mathbb{Q} : [K: \mathbb{Q}] = n, \lvert \Disc(K) \rvert \leq X \}$
count the number of degree $n$ number fields over $\mathbb{Q}$ with bounded discriminant of
size $X$. We consider all extensions as being inside a fixed algebraic closure
$\overline{\mathbb{Q}}$ of $\mathbb{Q}$.
A classical theorem of Hermite shows that $N_n(X)$ is finite, and an optimized and modern version of Hermite's argument due to Schmidt \cite{Schmidt} shows that $N_n(X)$ may be bounded by $O_n(X^{\frac{n+2}{4}})$.
It is conjectured that the actual sharp upper bound is $O_n(X)$; thus improvements to the Schmidt bound, and the classical work of Hermite, have attracted a lot of attention.

For $n \leq 5$, precise asymptotic formulas for $N_n(X)$ are known~\cite{DH69, CDDO02,
bhargavaquartic, bhargavaquintic}.  For large $n$, there are substantial improvements over the Schmidt bound, beginning with work of Ellenberg and Venkatesh~\cite{EllenbergVenkatesh}, which was subsequently improved in work of Couveignes~\cite{Couveignes} and Lemke Oliver and Thorne~\cite{LOT20}.  For $n \geq 95$, the best bound is that of~\cite{LOT20}, namely $N_n(X) \ll_n X^{c (\log n)^2}$,
where $c \leq 1.564$ is an explicitly computable constant.
While this improvement is substantial as $n \to \infty$, in the complementary regime $6 \leq n \leq 94$, the Schmidt bound has remained the best available for $25$ years.

Our main theorem gives a general improvement over the Schmidt bound.

\begin{theorem}\label{thm:main-bound}
For each $n\geq 6$ and every $\epsilon>0$, we have $N_n(X) \ll_{n,\epsilon} X^{\frac{n+2}{4} - \frac{1}{4n-4} + \epsilon}$.
\end{theorem}

We do not expect that the saving $\frac{1}{4n-4}$ is optimal for our method, but to extract a further saving from the work below would take significant care.  In fact, our work is inspired by work of Bhargava, Shankar, and Wang \cite{BSW} in a different context, and simultaneously with our work, an improvement of (essentially) $\frac{1}{2n-2}$ over Schmidt's bound was independently obtained by the same authors in \cite{BSW2}. Their independent work bears some similarity with our approach, but remarkably the two approaches differ greatly in multiple ways.  Therefore, we expedite some pieces of our approach, while focusing on other pieces where our techniques are different and which we expect to have other applications in different settings.

In \S\ref{sec:schmidt}, we summarize Schmidt's approach of estimating $N_n(X)$ by counting
monic polynomials $f(x) = x^n + c_1 x^{n-1} + \cdots + c_n$ where each $c_i$ is bounded by
$O(X^{\frac{i}{2n - 2}})$.
We slightly restructure Schmidt's method in order to not restrict to polynomials of trace $0$.
To improve upon Schmidt's bound, we separately count polynomials with ``particularly small''
discriminants and polynomials with ``particularly squarefull'' discriminants in comparison to the discriminants of the fields they cut out.
Our method is easier to describe in the context of Schmidt's method; thus, we defer further
description until \S\ref{ssec:strategy}.  Roughly speaking, we study polynomials with ``particularly small'' discriminant in \S\ref{sec:small-disc} and polynomials with ``particularly squarefull'' discriminant in \S\ref{sec:large-radical} and \S\ref{sec:strong+weak}.  Inspired by \cite{ShankarTsimerman-Heuristics}, we bound both the archimedean and nonarchimedean local density of polynomials of small discriminant in Proposition~\ref{prop:etale-density}; these densities are important in \S\ref{sec:large-radical} and \S\ref{sec:strong+weak}.  These densities bear some resemblance to the `Main Heuristic assumption' in \cite{ShankarTsimerman-Heuristics}, but Proposition~\ref{prop:etale-density}, though natural, appears to be novel.  Due to the significant blend of analytic techniques developed to work with polynomials in different contexts, we devote an early section, Section \ref{change of variables}, to their exposition.

It is in \S\ref{sec:nonarchimedean} and \S\ref{sec:large-radical} that our approach differs most substantially from \cite{BSW} and \cite{BSW2}.
In \S\ref{sec:nonarchimedean}, which we regard as our most novel, we study the nonarchimedean Fourier transform of the set of polynomials whose discriminant is divisible by a large power of a prime.
The main results of this section are somewhat technical statements about the support and size of this Fourier transform, but their proofs rely on a beautiful and substantial structure of the discriminant that is revealed upon taking its Fourier transform.
We expect that such Fourier transforms, and the underlying structure of the discriminant so revealed, may play a role in other problems beyond that considered in Theorem \ref{thm:main-bound}, for example perhaps in generalizing the ideas from \cite{ShankarTsimerman-Heuristics}.
We use the considerations of \S\ref{sec:nonarchimedean} to show in
\S\ref{sec:large-radical} bounds for polynomials whose discriminants are ``powerful": that is, they admit prime factors to large powers.  Section \ref{change of variables} describes the various techniques, such as change of variables,
that permeate the next three sections.  We hope this reference section facilitates the reading of the rest of the paper.
Finally, in \S\ref{sec:strong+weak}, we study polynomials whose discriminants have a large factor that is not powerful, incorporating the framework of~\cite{BSW}. The specific prime factorization of the large factor dividing the polynomial discriminant determines which polynomials are treated using \S\ref{sec:large-radical} or \S\ref{sec:strong+weak}, and is described in the next section.

\section{Outline of approach}\label{sec:schmidt}

Let $\mathcal{F}_n(X)$ denote the set of fields $K/\mathbb{Q}$ with degree $n$
and discriminant satisfying $\lvert \mathrm{Disc}(K) \rvert \leq X$.  Note that $N_n(X) = \#\mathcal{F}_n(X)$, our desired count.  The Schmidt bound \cite{Schmidt} asserts that $\#\mathcal{F}_n(X) \ll_n X^{\frac{n+2}{4}}$.  The key to Schmidt's approach is the following.

\begin{lemma} \label{lem:height-bound}
Let $K \in \mathcal{F}_n(X)$ and assume $K$ is primitive, i.e. that it has no
proper nontrivial subfields.  Then there is a monic polynomial $f(x) \in
\mathbb{Z}[x]$ with trace $0$, i.e. $f(x) = x^n + c_2 x^{n-2} + \cdots + c_n$,
for which $K \simeq \mathbb{Q}[x]/\langle f(x) \rangle$ and where each $c_i$
satisfies $\lvert c_i \rvert \ll_n X^{\frac{i}{2n-2}}$.
\end{lemma}
\begin{proof}
In the Minkowski embedding, the trace $0$ elements of $\mathcal{O}_K$ form a
rank $n-1$ lattice with covolume $\asymp_n \lvert \mathrm{Disc}(K)
\rvert^{1/2}$.  The $n-1$ Minkowski minima for this lattice satisfy $\lambda_1
\cdots \lambda_{n-1} \asymp \lvert \mathrm{Disc}(K) \rvert^{1/2}$, so
$\lambda_1 \ll_n \lvert \mathrm{Disc}(K) \rvert^{1/2(n-1)}$.  In particular,
there is a non-zero element $\alpha \in \mathcal{O}_K$ with trace $0$ whose
height is $O_n(\lvert \mathrm{Disc}(K) \rvert^{\frac{1}{2n-2}})$.  Since $\alpha$ has trace $0$ and $\alpha \neq 0$, it follows that $\alpha \not\in \mathbb{Q}$.  By our assumption that $K$ admits no interesting subfields, it follows that $\mathbb{Q}(\alpha) = K$.  The minimal polynomial of $\alpha$ then satisfies the conclusion of the lemma.
\end{proof}

There are $O_n(X^{\frac{n+2}{4}})$ polynomials of the type produced in Lemma \ref{lem:height-bound}.  It thus follows that the number of primitive fields contained in $\mathcal{F}_n(X)$ is also $O_n(X^{\frac{n+2}{4}})$.

\begin{remark}
To get a bound on non-primitive fields too, Schmidt (essentially) uses this argument, but over a moving subfield.  His bound actually improves for the number of imprimitive fields: if $d_{\max}$ is the largest proper divisor of $n$, his bound on imprimitive fields in $\mathcal{F}_n(X)$ is $X^{\frac{d_{\max}+2}{4}}$.  This is smaller than the bound claimed in Theorem \ref{thm:main-bound}, so it suffices to prove the claimed bound holds for the set $\mathcal{F}_n^\mathrm{prim}(X)$ of primitive field extensions.
\end{remark}

It is possible to rework Schmidt's argument slightly to make it marginally more convenient to invoke the results of the earlier work of Bhargava, Shankar, and Wang \cite{BSW}.  First, we consider polynomials without trace $0$.

Define the \emph{height} $H$ of a polynomial $f(x) = x^n + c_1 x^{n-1} + \dots
c_n$ by $H=\max\lvert c_i \rvert^{1/i}$.  We will frequently abuse notation by using the phrase `polynomials of height $H$' to mean polynomials with height bounded by $H$.
\begin{lemma} \label{lem:multiplicity}
Let $n \geq 2$.  There is a constant $C_n > 0$ such that each field $K \in
\mathcal{F}_n^\mathrm{prim}(X)$ is cut out by $\gg_n X^{\frac{1}{2n-2}}$
polynomials $f(x) \in \mathbb{Z}[x]$ of the form $f(x) = x^n + c_1 x^{n-1} +
\dots c_n$, where each $c_i$ satisfies $\lvert c_i \rvert \leq C_n^i X^{\frac{i}{2n-2}}$.
\end{lemma}

\begin{proof}
By the proof of Lemma \ref{lem:height-bound}, each $K \in \mathcal{F}_n^\mathrm{prim}(X)$
has an element $\alpha \in \mathcal{O}_K$ with trace $0$ and height bounded by $H_0 := B_n X^{\frac{1}{2n-2}}$,
for a positive constant $B_n$. Then there are at least $H_0$ elements in
$\mathcal{O}_K$ of height at most $2H_0$, corresponding to the translates
$\alpha + k$ for $\lvert k \rvert \leq H_0$, $k \in \mathbb{Z}$.  The minimal
polynomial $m(x) = x^n + c_1 x^{n-1} + \dots + c_n$ of such an element
satisfies $\lvert c_i \rvert \ll_n H_0^i$, and the result follows.
\end{proof}

For an admissible choice of $C_n$ in Lemma \ref{lem:multiplicity}, set $H := C_n X^{\frac{1}{2n-2}}$.
There are $O_n( H^{\frac{n^2+n}{2}})$ polynomials of height $H$, with each primitive field cut out by at least $\gg_n H$ different polynomials.
Dividing by this minimum multiplicity, we find that there must be no more than $O_n( H^{\frac{n^2+n-2}{2}}) = O_n(X^{\frac{n+2}{4}})$ different fields produced, again recovering Schmidt's bound.

\begin{remark}
The argument of Lemma \ref{lem:multiplicity} is a bit ad hoc, but its
conclusion is likely close to optimal, apart from constants depending on $n$.
For ``typical'' fields $K \in \mathcal{F}_n(X)$, we should expect the
multiplicity with which $K$ is cut out to be $\asymp_n H^n / \sqrt{ \lvert
\mathrm{Disc}(K) \rvert}$.  If $H = X^{\frac{1}{2n-2}} \approx \lvert
\mathrm{Disc}(K) \rvert^{\frac{1}{2n-2}}$, this expected multiplicity is about $H$.  The proof of Lemma \ref{lem:multiplicity} thus doesn't lose much, and its proof will be convenient in the next section.
\end{remark}

\subsection{Strategy and Proof of Main Theorem}\label{ssec:strategy}

While Lemmas \ref{lem:height-bound} and \ref{lem:multiplicity} guarantee that each field $K \in \mathcal{F}_n^\mathrm{prim}(X)$ is cut out by a polynomial of height $H \ll_n X^{\frac{1}{2n-2}}$, it is not the case that every such polynomial cuts out a field of discriminant at most $X$.
In particular, for ``typical'' $f(x)$ of height $H$, we have $\mathrm{disc}(f) \approx H^{n^2-n} = X^{\frac{n}{2}}$, and this is typically also the order of the discriminant of the field cut out by $f$.
We should thus expect the relevant polynomials attached to $K \in \mathcal{F}_n^\mathrm{prim}(X)$ to be exceptional in one of two ways: either the discriminant of $f$ is unusually small, or the discriminant of the field cut out by $f$ is much smaller than the discriminant of $f$.
In the latter case, the ratio of the two discriminants is the square of the index $[\mathcal{O}_K : \mathbb{Z}[\alpha]]$ where $\alpha$ is a root of $f$;
we call this integer the index of the polynomial $f$, denoted $\mathrm{index}(f)$.

The condition that the discriminant is small is global, while the condition that the index is large is local.  However, since both conditions are preserved under translation, it follows from the proof of Lemma \ref{lem:multiplicity} that each $K \in \mathcal{F}_n^\mathrm{prim}(X)$ will still be cut out by $\gg_n H$ of these exceptional polynomials.  Thus, our main task is to bound the number of exceptional polynomials.

\subsection*{Proof of Theorem~\ref{thm:main-bound}}

We now describe the proof of our main theorem more concretely.

By Lemma~\ref{lem:multiplicity}, to bound $K \in \mathcal{F}_n^{\mathrm{prim}}(X)$
it suffices to bound the number of polynomials $f$ with height up to
$H \asymp_n X^{\frac{1}{2n-2}}$ that cut out a field with discriminant $X \asymp_n H^{2n
-2}$, where each field is counted with multiplicity of order at least $H$.

In Section~\ref{sec:small-disc}, we prove the following corollary, bounding the the number of polynomials with small discriminant.

\begin{corollary} \label{cor:small-poly-disc}
Let $n \geq 3$ and let $H$ be sufficiently large in terms of $n$.
The number of polynomials $f(x) \in \mathbb{Z}[x]$ of the form $f(x)=x^n + c_1
x^{n-1} + \dots + c_n$ with $\lvert c_i \rvert \leq H^i$ and
$\lvert \disc(f) \rvert \leq H^{n^2 - n - 2}$ is $O_n(H^{\frac{n^2+n}{2}-1})$.
\end{corollary}

If $f(x)$ is irreducible and cuts out the field $K$, then
the discriminants of $f$ and $K$ are related through the equation
$\mathrm{disc}(f) = \mathrm{Disc}(K) [\mathcal{O}_K : \mathbb{Z}[\alpha]]^2$,
where $\alpha$ is a root of $f$ over $K$.  As indicated above, we refer to the index $[\mathcal{O}_K : \mathbb{Z}[\alpha]]$ as the index of $f$, which we denote by $\mathrm{index}(f)$.
Corollary \ref{cor:small-poly-disc} implies that with at most $O_n(H^{\frac{n^2+n}{2} - 1})$ exceptions,
\[
\mathrm{index}(f)^2 \cdot \mathrm{Disc}(K)
= \mathrm{disc}(f)
> H^{n^2 - n - 2}
.\]
As $\mathrm{Disc}(K) \leq X \asymp_n H^{2(n-1)}$,
each of these polynomials has large index bounded from below by
\begin{equation}\label{eq:index_bound}
\mathrm{index}(f)
\gg_n H^{\frac{n^2 - n - 2}{2} - (n-1)}
= H^{\frac{n(n-3)}{2}}.
\end{equation}

Thus each of these polynomials has discriminant divisible by a large square.
To bound the number of such polynomials, we consider two subcases: when the radical of $\idx(f)$ is small or large.  (Recall that the radical of an integer is the product of its prime divisors.)
In Section~\ref{sec:large-radical},
we prove the following bound on the number of polynomials with large index, but small index radical.

\begin{theorem}\label{thm:small-radical-index}
Let $n \geq 6$ and fix any $\epsilon > 0$.  Then for any $H \geq 1$, the number of polynomials $f(x) \in \mathbb{Z}[x]$ of degree $n$ and height $H$ for which $\mathrm{rad}(\mathrm{index}(f)) < H^{1-\epsilon}$
but $\mathrm{index}(f) > H^{\frac{n(n-3)}{2}}$ is $O_{n,\epsilon}(H^{\frac{n^2+n}{2} - \frac{4}{3} - \frac{4}{n} + \epsilon}+H^{\frac{n^2+n}{2}-\frac{2n}{3}+3+\epsilon})$.
\end{theorem}
Note that for all $n \geq 8$, the first term above is larger, and therefore governs our overall bound.  However, for $n=6,7$ the second term dominates.

In Section~\ref{sec:strong+weak}, we bound the number of polynomials with large index radical.

\begin{theorem}\label{thm:large-index}
For any $n \geq 3$, any $H \geq 1$, any $\epsilon > 0$, and any $M \geq 1$, the number of polynomials $f$ of degree $n$, height $H$, and where $m^2 \mid \disc{f}$ for some squarefree $m \geq M$ is
\begin{equation*}
  O_{n, \epsilon}\Big(%
    \frac{H^{\frac{n^2+n}{2}}}{\sqrt{M}}
    +
    H^{\frac{n^2+n}{2} - \frac{1}{2}+\epsilon}
  \Big).
\end{equation*}
\end{theorem}

Theorem~\ref{thm:small-radical-index} implies that the number of polynomials of height $H$, index bounded below by~\eqref{eq:index_bound}, and with $\rad(\idx(f)) < H^{1 - \epsilon}$ is $O_{n, \epsilon}(H^{\frac{n^2 + n}{2} - 1 + \epsilon})$.
All remaining uncounted polynomials have $\rad(\idx(f)) \geq H^{1 - \epsilon}$, and thus have a squarefree divisor $m$ of size at least $H^{1 - \epsilon}$ such that $m^2 \mid \disc(f)$.
Taking $M = H^{1 - \epsilon}$ in Theorem~\ref{thm:large-index}
shows that there are at most $H^{\frac{n^2 + n}{2} -\frac{1}{2} + \epsilon}$ such polynomials.
Combining these bounds, we find that
\begin{equation*}
H \cdot \# \mathcal{F}_n^{\mathrm{prim}}(X)
\ll_{n, \epsilon}
H^{\frac{n^2 + n}{2} - \frac{1}{2} + \epsilon},
\end{equation*}
which completes the proof of Theorem~\ref{thm:main-bound} using $H \asymp_n X^{\frac{1}{2n-2}}$.\qed{}

\section{Local fields, \'etale algebras and mass formulas}\label{sec:etale}
\label{change of variables}

Changes of variables between etale algebras, polynomial roots and polynomial coefficients will be used throughout this paper, so we introduce this section as a useful reference.  As such, the terminology here is self-contained.

Let $v$ be a place of $\mathbb{Q}$.  Throughout most of this paper, we will be concerned with the
properties of the set of monic, degree $n$ polynomials $f \in \mathbb{Z}[x]$ for which $\lvert
\mathrm{disc}(f) \rvert_v$ is small in some suitable sense.  We approach these questions locally,
viewing our polynomials (perhaps after a suitable change of variables) as lying inside
$\mathcal{O}_{\mathbb{Q}_v}[x]$, where $\mathcal{O}_{\mathbb{Q}_v} = \mathbb{Z}_v$ if $v$ is finite
and $\mathcal{O}_{\mathbb{Q}_v} = [-1,1]$ if $v = \infty$ (note that this selection of $[-1,1]$
instead of $\mathbb{R}$ may be nonstandard, but is useful in our context).  Rather than study the
discriminant of $f$ as a polynomial in its coefficients $\mathbf{c} \in
\mathcal{O}_{\mathbb{Q}_v}^n$, we prefer to study the discriminant by means of its simpler
expression in terms of the roots of $f$.  These roots may be thought of as an $n$-tuple
$\pmb{\lambda} \in \overline{\mathbb{Q}_v}^n$ for a fixed choice of algebraic closure
$\overline{\mathbb{Q}_v}$, but the locus of points inside $\overline{\mathbb{Q}_v}^n$ actually
arising as roots of polynomials in $\mathcal{O}_{\mathbb{Q}_v}[x]$ appears difficult to access
directly.  Thus, we pass through an intermediate space, the \'etale algebra associated with $f$.

Let $\mathcal{P}_n(\mathcal{O}_{\mathbb{Q}_v})$ be the space of monic, degree $n$ polynomials with
coefficients in $\mathcal{O}_{\mathbb{Q}_v}$.  Given a separable polynomial $f \in
\mathcal{P}_n(\mathcal{O}_{\mathbb{Q}_v})$, the quotient $\mathbb{Q}_v[x] / f$ is a degree $n$
\'etale algebra $K_v$ over $\mathbb{Q}_v$.  This \'etale algebra is equipped with a natural element
(namely, the image of $x$), and, conversely, given an element $\alpha \in K_v$, the polynomial
$\mathrm{Nm}(x-\alpha)$ will be associated with $\alpha$ under this map.  Now, for any $v$, there
are only finitely many degree $n$ \'etale algebras $K_v$ over $\mathbb{Q}_v$.  Thus, the image of
the map $f \mapsto x \in \mathbb{Q}_v[x]/f$ naturally decomposes as a disjoint union over the
different \'etale algebras.  Moreover, if $v$ is finite, then the image of $f$ in fact lands in
$\mathcal{O}_{K_v}$, the ring of integers of $K_v$.  If $v = \infty$, then we define
$\mathcal{O}_{K_v}$ to be the compact closure of the image.

Given $\alpha \in \mathcal{O}_{K_v}$, the roots of the associated polynomial $f_\alpha = \mathrm{Nm}(x-\alpha)$ are easily determined.  Explicitly, if we write $K_v = \oplus_{i=1}^r F_i$ with each $F_i$ a (field) extension of $\mathbb{Q}_v$, then any $\alpha \in \mathcal{O}_{K_v}$ may be expressed as $ \alpha = (\alpha_1,\dots,\alpha_r)$ with each $\alpha_i \in F_i$.  The roots of $f_\alpha$ are then the images of each $\alpha_i$ under the $[F_i : \mathbb{Q}_v]$ embeddings of $F_i$ into our fixed choice $\overline{\mathbb{Q}_v}$.  Thus, while we are in principle interested in performing computations in the space of roots $\pmb{\lambda} \in \overline{\mathbb{Q}_v}^n$, in practice, we instead perform these computations inside \'etale algebras.

In carrying out the computations to follow, we find it convenient to fix a coordinatization of $\mathcal{O}_{K_v}$ when $v$ is finite.  In particular, if $K_v = \oplus_{i=1}^r F_i$ as above, then by choosing an integral basis $\omega_{i,1},\dots,\omega_{i,[F_i:\mathbb{Q}_v]}$ for each $F_i$, the collection $\omega_{1,1},\dots,\omega_{r,[F_r:\mathbb{Q}_v]}$ is an integral basis for $K_v$. This identifies $\mathcal{O}_{K_v}$ with $\mathcal{O}_{\mathbb{Q}_v}^n$, whose generic element we denote $\mathbf{a}$.  The transformation from the coordinates $\mathbf{a}$ to the roots $\pmb{\lambda}$ is linear, with block diagonal matrix $M_{K_v} = \oplus_{i=1}^r M_{F_i}$, where
    \[
        M_{F_i} = \begin{pmatrix} \iota_{1}(\omega_{i,1}) & \dots & \iota_{1}(\omega_{i,[F_i:\mathbb{Q}_v]}) \\
            \vdots & & \vdots \\
            \iota_{[F_i:\mathbb{Q}_v]}(\omega_{i,1}) & \dots & \iota_{[F_i:\mathbb{Q}_v]}(\omega_{i,[F_i:\mathbb{Q}_v]}) \end{pmatrix},
    \]
with $\iota_1,\dots,\iota_{[F_i:\mathbb{Q}_v]}$ denoting the embeddings of $F_i$ into
$\overline{\mathbb{Q}_v}$.  Notice that $\lvert \det M_{F_i} \rvert_v = \lvert \mathrm{disc} (F_i)
\rvert_v^{1/2}$ by definition of the discriminant, hence $\lvert \det M_{K_v}\rvert_v = \lvert
\mathrm{disc} (K_v) \rvert_v^{1/2}$.

Closing the loop, the roots $\pmb{\lambda}$ of a polynomial $f \in \mathcal{P}_n(\mathcal{O}_{\mathbb{Q}_v})$ naturally determine its coefficients $\mathbf{c}$ by means of the elementary symmetric polynomials.   Namely, given the roots $\pmb{\lambda}=(\lambda_1, \dots, \lambda_n) \in
\overline{\mathbb{Q}_v}$ of a monic polynomial $f$, define the change of variable $\pmb{\lambda}$ to $\pmb{\sigma}=(\sigma_1,\cdots, \sigma_n) \in \overline{\mathbb{Q}_v}$ via
\begin{equation*}
\sigma_i(\pmb{\lambda})
=
\sum_{\substack{S\subset \{1, \cdots, n \} \\ \lvert S \rvert = i}} \prod_{j\in S} \lambda_j
\end{equation*}
for $i=1,\dots,n$.  This change of variable yields the polynomial $f(x)=x^n-\sigma_1 x^{n-1}+\cdots (-1)^k\sigma_k x^{n-k}
+\cdots +  (-1)^n\sigma_n$.  Each $\sigma_i$ is the $i$-th elementary symmetric polynomial in $\lambda_i$.

Figure \ref{fig:regime-change} summarizes the different spaces, maps between them, and variable naming conventions we adopt.

\begin{figure}[h]
\begin{tikzcd}
&  & \begin{array}{c} \text{\'Etale algebras} \\ \displaystyle \coprod_{K_v/\mathbb{Q}_v \text{ deg $n$}} \mathcal{O}_{K_v} \ni \alpha \end{array} \arrow[lldd, "\phi \colon \alpha \mapsto \mathrm{Nm}(x-\alpha)=:f_\alpha", bend left] \arrow[rrdd] &  & \begin{array}{c} \text{Coordinates ($v=p$)} \\ \displaystyle \coprod_{K_v/\mathbb{Q}_v \text{ deg $n$}} \mathcal{O}_{\mathbb{Q}_v}^n \ni \mathbf{a} \end{array} \arrow[ll, "\sim"'] \arrow[dd, "\coprod M_{K_v}"]\\
&  &  &  &  \\
\begin{array}{c} \text{Polynomials} \\ \mathcal{P}_n(\mathcal{O}_{\mathbb{Q}_v}) \ni f \end{array} \arrow[rruu, "{f \mapsto x \in \mathbb{Q}_v[x]/f}", dashed, bend left] &  &  &  & \begin{array}{c} \text{Roots} \\ \overline{\mathbb{Q}_v}^n \ni \pmb{\lambda} \end{array} \arrow[llll, "{\pmb{\lambda} \mapsto \prod(x-\lambda_i) = \sum (-1)^i \sigma_i(\pmb{\lambda}) x^{n-i} }", bend left] \\
& & & & \\
\begin{array}{c} \text{Coefficients} \\ \mathcal{O}_{\mathbb{Q}_v}^n \ni \mathbf{c} \end{array} \arrow[uu, "\mathbf{c} \mapsto f_\mathbf{c}"] &&&&
\end{tikzcd}

\begin{caption}
    { \label{fig:regime-change} Maps between spaces of polynomials, \'etale algebras, and roots.  The dashed arrow indicates this map is defined for separable polynomials.}
\end{caption}
\end{figure}

At this stage, it only remains to discuss the change of measure between the various spaces, in particular between polynomials and \'etale algebras.  First, the space $\mathcal{P}_n(\mathcal{O}_{\mathbb{Q}_v})$ is equipped with a natural Haar/Lebesgue measure $\nu$ by means of the coefficient isomorphism $\mathcal{O}_{\mathbb{Q}_v}^n \to \mathcal{P}_n(\mathcal{O}_{\mathbb{Q}_v})$; if $v$ is finite, we normalize the measure $\nu$ so that $\nu(\mathcal{P}_n(\mathcal{O}_{\mathbb{Q}_v})) = 1$.  Each \'etale algebra $K_v$ is also equipped with a natural Haar/Lebesgue measure that we denote $\mu$, normalized if $v$ is finite so that $\mu(\mathcal{O}_{K_v}) = 1$.  The fundamental result we use is due to Serre \cite[Lemma~3]{Serre} and Shankar--Tsimerman \cite[Lemma~2.2]{ShankarTsimerman-Heuristics}.

\begin{lemma}
    With notation as above, we have $\phi^* \nu = \lvert \mathrm{disc}(K_v) \rvert_v^{1/2} \lvert
    \mathrm{disc}(f_\alpha) \rvert_v^{1/2} \mu$.  In particular, for any $\nu$-integrable function $\psi$ on $\mathcal{P}_n(\mathcal{O}_{\mathbb{Q}_v})$, we have
        \[
            \int_{\mathcal{P}_n(\mathcal{O}_{\mathbb{Q}_v})} \psi(f) \, d\nu
                = \sum_{K_v/\mathbb{Q}_v \text{ deg. n}} \frac{\lvert \mathrm{disc}(K_v)
                \rvert_v^{1/2}}{\lvert \mathrm{Aut}(K_v) \rvert} \int_{\mathcal{O}_{K_v}} \lvert
                \mathrm{disc}(f_\alpha) \rvert_v^{1/2} \psi(f_\alpha) \, d\mu.
        \]
\end{lemma}

\section{The density of polynomials with small discriminant}\label{sec:small-disc}

We begin by proving the following proposition on the density of polynomials over some completion $\mathbb{Q}_v$ of $\mathbb{Q}$ with small discriminant.  This result in the case that $\mathbb{Q}_v = \mathbb{R}$ will be the key ingredient in the proof of Corollary \ref{cor:small-poly-disc}, and the ideas behind the proof in the case that $\mathbb{Q}_v$ is non-archimedean will be the key ingredient in the proof of Theorem \ref{thm:small-radical-index}.

\begin{proposition}\label{prop:etale-density}
Let $n \geq 2$ and let $\mathbb{Q}_v$ be a completion of $\mathbb{Q}$.  Let $\lvert \cdot \rvert_v$ be the associated absolute value and let $\mu_v$ be the associated Haar measure on $\mathbb{Q}_v$, normalized in the case that $v$ is finite so that the total measure of $\mathbb{Z}_v$ is $1$ and so that $\mu_v$ agrees with usual Lebesgue measure in the case that $v$ is infinite.  For $\mathbf{c} \in \mathbb{Q}_v^n$, let $f_\mathbf{c}(x) := x^n + c_1x^{n-1} + \dots + c_n$.

Then for any $\delta \in (0,1)$, there holds
\begin{equation}\label{bound:small_disc}
\nu( \{ \mathbf{c} \in \mathbb{Q}_v^n : \lvert c_i \rvert_v \leq 1 \text{ for all } i \text{ and }
\lvert \disc(f_\mathbf{c}(x)) \rvert_v \leq \delta \})
\ll_n \delta^{\frac{1}{2}+\frac{1}{n}},
\end{equation}
where $\nu$ denotes the product measure on $\mathbb{Q}_v^n$.
\end{proposition}

\begin{proof}
Let $\mathbf{1}_{\delta}(f)$ be the characteristic function for the set of polynomials where $\lvert
\mathrm{Disc}(f) \rvert_v \leq \delta$.

Then applying the change of variables discussed in Section \ref{change of variables}, we find
\begin{align*}
\nu( \{ \mathbf{c} \in \mathbb{Q}_v^n :& \lvert c_i \rvert_v \leq 1 \text{ for all } i \text{ and }
\lvert \disc(f_\mathbf{c}(x)) \rvert_v \leq \delta \})
\\&=
\sum_{[{K_v}:\mathbb{Q}_v]=n} \frac{\lvert \Disc({K_v}) \rvert_{v}^{1/2}}{\lvert \mathrm{Aut}({K_v})
\rvert} \int_{\mathcal{O}_{{K_v}}} \lvert \Disc(f_\alpha) \rvert_v^{1/2} \mathbf{1}_{\delta}(f_\alpha) \dd{\mu}(\alpha)
\\ &\leq
{\delta^{1/2}} \sum_{[{K_v}:\mathbb{Q}_v]=n} \frac{\lvert \Disc({K_v}) \rvert_{v}^{1/2}}{\lvert
\mathrm{Aut}({K_v}) \rvert} \int_{\mathcal{O}_{{K_v}}} \mathbf{1}_{\delta}(f_\alpha) \dd{\mu}(\alpha)
,\end{align*}
where $f_\alpha$ denotes the characteristic polynomial of $\alpha$, and where,
in the inequality in the second line, we used the bound $\lvert \Disc(f_\alpha) \rvert_v = \lvert \Disc(\alpha)\rvert_v \leq \delta$, valid for any $\alpha$ for which $\mathbf{1}_{\delta}(f_\alpha) \neq 0$.  Note in particular that any polynomial with a repeated root has discriminant $0$, but that the set of such polynomials has measure $0$ and thus does not affect the integral. \

We now estimate the integral, beginning with the case where $K_v \simeq \mathbb{Q}_v^n$ is the
totally split algebra.  In this case, $\lvert \mathrm{Disc}(\alpha) \rvert_v = \prod_{i < j} \lvert
\alpha_i-\alpha_j \rvert_v^2$, where $\alpha_i, \alpha_j$ are the roots of $f_\alpha$.  Consequently, for any point in the support of $\mathbf{1}_\delta$, we have
\[
\prod_{i=1}^n \prod_{j \ne i} \lvert \alpha_i-\alpha_j \rvert_v \leq \delta
.\]
The pigeonhole principle then implies that there must be some $i$ for which
\[
    \prod_{j \ne i} \lvert \alpha_i - \alpha_j \rvert_v \leq \delta^{1/n}.
\]
If we let $\alpha_j^\prime = \alpha_i-\alpha_j$ for each $j\neq i$, it thus follows that $\prod_{j
\ne i} \lvert \alpha_j^\prime \rvert_v \leq \delta^{1/n}$.  Since $\mathcal{O}_{K_v}$ is compact, the measure of points satisfying this condition is $O_n(\delta^{1/n})$, independent of $\alpha_i$.  We therefore conclude in this case that
\[
\int_{\mathcal{O}_{K_v}} \mathbf{1}_\delta(f_\alpha)\dd\mu(\alpha)
\ll_n \delta^{\frac{1}{n}}.
\]

We now turn to the case that $K_v$ is not the totally split algebra.  The idea in this case is substantially the same, but  requires the notation from Section 3.  The discriminant $\mathrm{Disc}(\alpha)$ satisfies
\[
\lvert \mathrm{Disc}(\alpha) \rvert_v = \prod_{i=1}^n\prod_{j \ne i} \lvert \lambda_i-\lambda_j
\rvert_v.
\]
Proceeding as in the totally split case, we conclude that for $\alpha$ in the support of $\mathbf{1}_\delta$, there must be some $i$ for which
\begin{equation}\label{bound:php}
\prod_{j\ne i} \lvert \lambda_i-\lambda_j \rvert_v \leq \delta^{\frac{1}{n}}.
\end{equation}
Accounting for the determinant of the matrix $M_{K_v}$ sending the coordinates $\mathbf{a}$ to the roots $\pmb{\lambda}=(\lambda_1,\dots,\lambda_n)$, we conclude that
\[
\int_{\mathcal{O}_{K_v}} \mathbf{1}_{\delta}(f_\alpha) \dd \mu(\alpha)
\ll_n \lvert \mathrm{Disc}(K_v) \rvert_v^{-1/2} \delta^{\frac{1}{n}}.
\]
We therefore find that
\[
{\delta^{1/2}} \sum_{[{K_v}:\mathbb{Q}_v]=n} \frac{\lvert \Disc({K_v}) \rvert_{v}^{1/2}}{\lvert
\mathrm{Aut}({K_v}) \rvert} \int_{\mathcal{O}_{{K_v}}} \mathbf{1}_{\delta}(f_\alpha) \dd{\mu}(\alpha)
\ll_n \delta^{\frac{1}{2}+\frac{1}{n}} \sum_{[{K_v}:\mathbb{Q}_v]=n} \frac{1}{\lvert
\mathrm{Aut}({K_v}) \rvert}
\ll_n \delta^{\frac{1}{2}+\frac{1}{n}},
\]
since the summation may be bounded in terms of the number of \'etale algebras of degree $n$, which in turn may be bounded solely in terms of $n$.
\end{proof}

\begin{remark}
The bound \eqref{bound:small_disc} is sharp for an infinite sequence of $\delta$ tending to 0. In particular, one can realize a lower bound of magnitude $\gg \delta^{\frac{1}{2}+\frac{1}{n}}$ for all $\delta$ when $v$ is an archimedean valuation.
One can also realize this lower bound for the nonarchimedean totally split \'etale algebras $K_v \simeq \mathbb{Q}_v^n$ and $\delta = p^{-k}$ when $k = mn(n-1)$ for some $k$ and $m \geq 1$ since the product in \eqref{bound:php} has $n(n-1)$ factors in these cases, and therefore can be realized in these cases.
For many other $\delta$, the bound can be improved using the discrete nature of the valuation group. In our optimization process below, we have no freedom in choosing $\delta$ and cannot exploit such improvements. Therefore, we rely on the upper bound \eqref{bound:small_disc}.
\end{remark}

We now use Proposition \ref{prop:etale-density} to bound the number of integral polynomials with small archimedean discriminant.

\begin{lemma}\label{lem:small-poly-disc}
Let $n \geq 3$, $Y \geq 1$, and $H$ sufficiently large in terms of $n$.
Then the number of polynomials $f(x) \in \mathbb{Z}[x]$ of the form
$f(x) = x^n + c_1 x^{n-1} + \dots + c_n$ with $\lvert c_i \rvert \leq H^i$
and $\mathrm{disc}(f) \leq H^{n^2-n} / Y$ is
$O_n(H^{\frac{n^2+n}{2}}/Y^{\frac{1}{2}+\frac{1}{n}} + H^{\frac{n^2+n}{2}-1})$.
\end{lemma}

The key idea behind Lemma \ref{lem:small-poly-disc} is that the coefficients of the polynomials $f(x)$ with small discriminant must lie in a compact region in $\mathbb{R}^n$ with small volume, and thus we should expect there to be relatively few polynomials whose coefficients lie in this region.
This is made rigorous by means of the Lipschitz principle from \cite{Davenport}:

\begin{lemma}[Davenport]\label{lem:davenport}
If $\Omega \subseteq \mathbb{R}^n$ is a  compact semialgebraic region (i.e., cut out by algebraic inequalities), then the number of lattice points in the intersection satisfies the bound
\begin{equation}\label{eq:counttovolume}
\mathbb{Z}^n \cap \Omega
= \mathrm{vol}(\Omega) + O(\max_{\pi} \mathrm{vol}\big(\pi(\Omega))\big),
\end{equation}
where the maximum runs over the projections $\pi$ of $\mathbb{R}^n$ onto its various coordinate hyperplanes (i.e., the regions in $\mathbb{R}^d$ for $d<n$ obtained by ``forgetting'' $n-d$ of the coordinates).
The implicit constants depend only on the dimension $n$ and the degrees of the equations defining $\Omega$.
\end{lemma}

\begin{proof}[Proof of Lemma \ref{lem:small-poly-disc}]
Define $\Omega_{H,Y} \subseteq \mathbb{R}^n$ to be
\begin{equation*}
\Omega_{H,Y} :=
\{ (c_1,\dots,c_n) \in \mathbb{R}^n :
\lvert c_i \rvert \leq H^i,
\mathrm{disc}(x^n + c_1 x^{n-1} + \cdots + c_n) \leq H^{n^2-n}/Y
\}.
\end{equation*}
The maximum volume of the coordinate projections of $\Omega_{H,Y}$ is trivially $O_n(H^{\frac{n^2+n}{2}-1})$, coming from forgetting the discriminant condition and using the coordinate projection $(c_1,\dots,c_n) \mapsto (c_2,\dots,c_n)$.
By Lemma \ref{lem:davenport}, the statement thus reduces to showing that $\mathrm{vol}(\Omega_{H,Y}) \ll_n H^{\frac{n^2+n}{2}} / Y^{\frac{1}{2} + \frac{1}{n}}$.
Since the region $\Omega_{H,Y}$ is obtained by an anisotropic dilation we find $\mathrm{vol}(\Omega_{H,Y}) = H^{\frac{n^2+n}{2}} \mathrm{vol}(\Omega_{1,Y})$. Thus it suffices to bound $\mathrm{vol}(\Omega_{1, Y})$.
Taking $\delta=Y^{-1}$ and $\mathbb{Q}_v=\mathbb{R}$ in Proposition~\ref{prop:etale-density}, we immediately deduce that
\begin{equation}
\mathrm{vol}(\Omega_{1, Y}) \ll_n Y^{-\frac{1}{2}-\frac{1}{n}}.
\end{equation}
This completes the proof.
\end{proof}

This lemma implies Corollary~\ref{cor:small-poly-disc}, used in \S\ref{ssec:strategy} to prove our main theorem.
\begin{proof}[Proof of Corollary~\ref{cor:small-poly-disc}]
  Take $Y = H^{2}$ in Lemma~\ref{lem:small-poly-disc}.
\end{proof}

\section{Nonarchimedean Fourier transforms}\label{sec:nonarchimedean}

Recall that if $f(x)$ is irreducible and cuts out the field $K$, then its discriminant satisfies
$\mathrm{disc}(f) = \mathrm{Disc}(K) [\mathcal{O}_K : \mathbb{Z}[\alpha]]^2$,
where $\alpha$ is a root of $f$ over $K$, and where $[\mathcal{O}_K : \mathbb{Z}[\alpha]] =: \idx(f)$.  In this section, we assemble the main technical ingredients that will be used in the proof of Theorem~\ref{thm:small-radical-index}.  This theorem bounds the number of integral polynomials $f$ whose index is large, but for which $\mathrm{rad}(\idx(f))$ is small.  This condition implies that such polynomials will have discriminants divisible by large powers of primes.  We therefore begin with the following simple lemma.

\begin{lemma}\label{lem:discriminant-classes}
Let $n \geq 2$ and $k \geq 1$, and let $p$ be prime.  The set of monic polynomials $f \in \mathbb{Z}[x]$ for which $p^{2k}$ divides the discriminant of $f$ is determined by congruence conditions $\pmod{p^{2k}}$ with relative density $\ll p^{-k-\frac{2k}{n}}$.
\end{lemma}
\begin{proof}
    The fact that this condition is determined from congruence conditions follows from the fact that the discriminant of a polynomial is a polynomial in its coefficients.  The density of these congruence conditions may be determined by computing the Haar measure of the set of monic polynomials in $\mathbb{Z}_p[x]$ satisfying this condition.  Therefore, the claim about the density follows from Proposition~\ref{prop:etale-density} by taking $\mathbb{Q}_v = \mathbb{Q}_p$ and $\delta = p^{-2k}$.
\end{proof}

We now introduce the main object of study in this section.  Let $\mathcal{P}_n(\mathbb{Z}/p^{2k} \mathbb{Z})$ be the set of monic degree $n$ polynomials over $\mathbb{Z}/p^{2k}\mathbb{Z}$, and let $\psi_{p^{2k}}$ be the characteristic function of the subset of polynomials whose discriminant is congruent to $0 \pmod{p^{2k}}$.
For any $\mathbf{u} \in \mathbb{Z}^n$, we define the Fourier transform $\widehat{\psi_{p^{2k}}}(\mathbf{u})$ of $\psi_{p^{2k}}$ by
    \begin{equation}\label{eqn:fourier-def}
        \widehat{\psi_{p^{2k}}}(\mathbf{u})
            := \frac{1}{p^{2kn}} \sum_{\mathbf{c} \in (\mathbb{Z}/p^{2k}\mathbb{Z})^n} \psi_{p^{2k}}(f_\mathbf{c}) \exp\left(\frac{ 2\pi i \langle \mathbf{c},\mathbf{u}\rangle}{p^{2k}} \right),
    \end{equation}
where $f_\mathbf{c}(x) = x^n + c_1 x^{n-1} + \dots + c_n$ and $\langle \cdot,\cdot\rangle$ denotes the standard inner product.

We begin with a series of lemmas that will be used to determine the support of $\widehat{\psi_{p^{2k}}}$.

\begin{lemma}\label{lem:relation}
Write the monic polynomial $f$ as $f(x)=x^n-\sigma_1 x^{n-1}+\cdots (-1)^k\sigma_k x^{n-k}
+\cdots +  (-1)^n\sigma_n$.
For each $1 \leq i \leq n$, define $D_i:=\frac{\partial \disc{f}}{\partial \sigma_i} \in
\mathbb{Z}[\sigma_1, \ldots, \sigma_n]$.
Then the discriminant $\disc(f)$ and partial derivatives $D_i$ satisfy
\begin{enumerate}
\item
For any $1\le r\le r+k\le n$ and $1\le s-k\le s\le n$, we have $\disc(f) \mid (D_r D_s-D_{r+k}D_{s-k})$.
\item
$\sum_{1\le i\le n} D_i\cdot (n+1-i)\cdot \sigma_{i-1} = 0$ where we define $\sigma_0 = 1$.
\end{enumerate}
\end{lemma}
The second part of the above lemma can be derived from equation (1.27) in Chapter 12 of \cite{MR2394437}, but we were unaware of this reference until recently, so we provide a proof.

\begin{proof}
Our first goal is to give an expression for
$D_i(\pmb{\sigma}(\pmb{\lambda})) \in \mathbb{C}[\pmb{\lambda}]$.
Recall the change of variables to go from roots $\pmb{\lambda}$ to coefficients $\pmb{\sigma}(\pmb{\lambda})$ of a polynomial described in Section \ref{change of variables}.   The Jacobian matrix for this change of variables is $B(\pmb{\lambda})$ where
$B_{ij} = \frac{\partial{c_i}}{\partial \lambda_j}$.
A quick computation verifies that
\begin{equation*}
\frac{\partial{c_i}}{\partial \lambda_j}
=
c_{i-1}(\lambda_1, \ldots,
\hat{\lambda_j},
\ldots, \lambda_n),
\end{equation*}
where we write $\hat{\lambda_j}$ to indicate that the $j$-th coordinate should be omitted.
Below, we use the shorthand
$\hat{\pmb{\lambda}_j} = (\lambda_1, \ldots, \hat{\lambda_j}, \ldots, \lambda_n)$.  When $\lambda_i \neq \lambda_j$ for all $i\neq j$, we define the matrix
$A(\pmb{\lambda}) = (A_{ij})$ by
\begin{equation*}
A_{ij} = \lambda_i^{n-j} (-1)^{j+1} / \prod_{k\neq i}(\lambda_i-\lambda_k).
\end{equation*}
Then one can verify that
\begin{equation*}
(AB)_{ij} \prod_{k\neq i}(\lambda_i - \lambda_k)
=
\big(%
\lambda_i^{n-1}\sigma_0(\hat{\pmb{\lambda}_j})
- \lambda_i^{n-2}\sigma_{1}(\hat{\pmb{\lambda}_{j}})
+ \cdots
+ (-1)^{n-1}\sigma_{n-1}(\hat{\pmb{\lambda}_{j}})
\big)
= \prod_{k \neq j} (\lambda_i - \lambda_k).
\end{equation*}
To see the last equality, we used the identity
\begin{equation}\label{eq:symmetric_polys}
\prod_{i=1}^{n-1} (x - a_i) = \sum_{k=0}^{n-1} (-1)^k \sigma_i(\pmb{a}) x^{n-1-k}
\end{equation}
for expressing a polynomial with roots
$\pmb{a} = (a_1, \ldots, a_{n-1})$ in terms of the elementary symmetric polynomials in these roots.
This shows that $(AB)_{ij}$ is $0$ if $i \neq j$ and is equal to $1$ if $i = j$, hence
$AB = I$.

By the inverse function theorem and chain rule, we obtain the inverse function $\lambda_i=
\lambda_i(\pmb{\sigma})$ with
$\frac{\partial \lambda_i }{\partial \sigma_j } = A_{ij}$, and
$\frac{\partial D}{\partial \sigma_i}
= \sum_{j} \frac{\partial D}{\partial \lambda_j} \frac{\partial \lambda_j}{\partial \sigma_i}
= \sum_j \frac{\partial D}{\partial \lambda_j} A_{ji}$.

By an abuse of notation, we write $\disc(\mathbf{c}) := D(\mathbf{c}) := \disc(f_{\mathbf{c}})$,
and with $f = x^n - \sigma_1 x^{n-1} + \cdots (-1)^n \sigma_n$ as in the lemma statement, we
write $D = \disc(f)$.
For $\pmb{\sigma} = (\sigma_1, \ldots, \sigma_n) = \pmb{\sigma}(\pmb{\lambda})$, we note
that~\eqref{eq:symmetric_polys} also shows that the roots of $f$ are precisely the $\lambda_i$.
Thus $D(\pmb{\sigma}(\pmb{\lambda})) = \prod_{i<j} (\lambda_i-\lambda_j)^2$, and we have
$\frac{\partial D}{\partial \lambda_j} = 2D\sum_{k\neq j} (\lambda_j-\lambda_k)^{-1}$.
This implies that
\begin{equation*}
D_i
:=
\frac{\partial D}{\partial \sigma_i}
=
2D \sum_{j} \lambda_j^{n-i}(-1)^{i+1}
\cdot
\prod_{k\neq j}(\lambda_j-\lambda_k)^{-1}
\cdot \big( \sum_{k\neq j} (\lambda_j-\lambda_k)^{-1} \big)
= 2D\sum_{j} \lambda_j^{n-i}(-1)^{i+1}\cdot W_j,
\end{equation*}
where we denote
\begin{equation*}
W_j = \prod_{k\neq j}(\lambda_j-\lambda_k)^{-1} \cdot \sum_{k\neq j} (\lambda_j-\lambda_k)^{-1}.
\end{equation*}
Observe that
\begin{align*}
DW_j
&=
\prod_{i<k,i\neq j,k\neq j} (\lambda_i-\lambda_k)^2
\cdot
\prod_{k\neq j}(\lambda_j-\lambda_k)
\cdot
\sum_{k\neq j} (\lambda_j-\lambda_k)^{-1}
\end{align*}
is a polynomial in $\mathbb{C}[\pmb{\lambda}]$.
We thus obtain $D_i(\pmb{\lambda})$ as a polynomial in the Zariski open set
$D(\pmb{\lambda})\neq 0$ in the affine space of $\pmb{\lambda}$.
As $D_i(\pmb{\sigma}(\pmb{\lambda}))$ agrees with $D_i(\pmb{\lambda})$ in that open set, they
agree everywhere in $\pmb{\lambda}$ affine space (and not merely when $\lambda_i \neq
\lambda_j$ for all $i \neq j$).

We now evaluate $D_i$. Notice that $DW_j$ is zero if $\lambda_j$ is a triple root of $f$,
or if there exists a double root $\lambda_i = \lambda_k$ with $i, k \neq j$.
If there is a unique double root $\lambda_j = \lambda_{k_0}$, then
$DW_j = \prod_{i<k,i\neq j,k\neq j} (\lambda_i-\lambda_k)^2
\prod_{k\neq j, k_0} (\lambda_j-\lambda_k) = DW_{k_0}$.
Stated differently --- if $f$ has two distinct pairs of double roots or a triple
root, then $D_i = 0$ for all $i$. If there is a unique double root $\lambda_j$, then for any
$i$ we compute
\begin{equation*}
D_i =(-1)^{i+1} \lambda_j^{n-i}\cdot 4DW_j.
\end{equation*}
Therefore as a polynomial in $\mathbb{C}[\pmb{\sigma}]$, it follows that
\begin{equation*}
D_r D_s-D_{r+k}D_{s-k}
\equiv
(4DW_j)^2(-1)^{r+s}\cdot (\lambda_{j}^{2n-r-s}-\lambda_j^{2n-r-s})
=
0
\end{equation*}
on the variety cut out by $D(\pmb{\sigma})=0$. Since $D(\pmb{\sigma})$ is irreducible, it
follows from Hilbert's Nullstellensatz that $D \mid D_r D_s-D_{r+k}D_{s-k}$.

In order to prove the second statement, notice that discriminant is translation invariant.
That is,
$\disc(\lambda_1+t,\lambda_2+t,\ldots, \lambda_n+t)
= \disc(\lambda_1,\lambda_2,\ldots, \lambda_n)$.
Now we can consider the function $\sigma_i(t):= \sigma_i(\lambda_1+t,\lambda_2+t,\ldots, \lambda_n+t)$.
It is clear that $\frac{d \sigma_i(t)}{dt} = \sigma_{i-1}(\lambda_1+t,\ldots, \lambda_n+t) (n-i+1)$.
Therefore
\begin{equation*}
\left.\frac{d D(\pmb{\sigma}(t))}{dt}\right|_{t=0} = \left. \sum_{i}\frac{\partial D}{\partial \sigma_i}
\frac{d \sigma_i}{ d t}\right|_{t=0} = \sum_i D_i \cdot \sigma_{i-1}\cdot (n-i+1) = 0,
\end{equation*}
completing the proof.
\end{proof}

This lemma is general. We exploit the relationship $D \mid D_r D_s - D_{r+k} D_{s-k}$ below, but we note that there are many other algebraic relationships on the discriminant and its derivatives that might yield further refinements on the structure of the support of $\psi_{p^{2k}}$.
In practice, we use the following more specific lemma.

\begin{lemma}\label{lem:arithmetic_progression_disc}
Given $\mathbf{c} \in \mathbb{Z}^n$ corresponding to a polynomial $f_{\mathbf{c}}(x)$ with
$\disc(f_{\mathbf{c}}) \equiv 0 \bmod p^{2k}$, let $\mathbf{D}_{\mathbf{c}} :=
(\frac{\partial \disc(f)}{\partial c_1}, \ldots, \frac{\partial \disc(f)}{\partial c_n})$
denote the gradient vector of the discriminant function, and let $v_i(\mathbf{c})$ denote the
valuation at $p$ of $\frac{\partial \disc(f)}{\partial c_i} (\mathbf{c}) = D_i(\mathbf{c})$.

Then there either exists $a \in \mathbb{Z}_{\geq 0}$ such that
\begin{equation}\label{APeqn}
\min\{ v_i(\mathbf{c}), k \} = \min \{v_n(\mathbf{c}) + (n-i) a, k \}
\end{equation}
or $b \in \mathbb{Z}$ with $0 \leq b \leq \min(\mathrm{val}_p(n), k)$ such that
\begin{equation}\label{APeqn_increasing}
\min\{ v_i(\mathbf{c}), k \} = \min \{v_1(\mathbf{c}) + (i-1) b, k \}.
\end{equation}
\end{lemma}

Intuitively, $\min\{ v_i(\mathbf{c}), k\}$ almost forms an arithmetic
progression, except that when a term is greater than $k$ we change it to $k$.
Thus this lemma shows that the support of $\psi_{p^{2k}}$
is constrained to these ``near arithmetic progressions''.

\begin{proof}
By Lemma~\ref{lem:relation} (1), we see that when $p^{2k} \mid \disc(f_{\textbf{c}})$, $p^{2k} \mid
D_{r}(\textbf{c})D_{s}(\textbf{c}) - D_{r+\ell}(\textbf{c}) D_{s-\ell}(\textbf{c})$ for all relevant
$r, s, \ell$.
This implies the valuation relations $\min \{v_{r} (\mathbf{c})+v_{s} (\mathbf{c}), 2k\} =
\min\{v_{r+\ell} (\mathbf{c})+v_{s-\ell} (\mathbf{c}), 2k\}$. Specializing to $r = s = i$ and $\ell
= 1$ gives that $\min\{2v_{i}(\mathbf{c}), 2k\} = \min\{v_{i-1}(\mathbf{c}) + v_{i+1}(\mathbf{c}),
2k\}$.
Thus $\big(v_{i-1}(\mathbf{c}), v_i(\mathbf{c}), v_{i+1}(\mathbf{c})\big)$ is an arithmetic
progression when $v_{i}(\mathbf{c}) < k$ and ``nearly'' an arithmetic progression otherwise.

It is not possible for the sequence $\min\{v_i(c), k\}$ to initially decrease and later decrease.
For otherwise, there would exist $r$ and $s$, with $r < s - 1$, such that $\min\{v_r(c), k\}$ and
$\min\{v_s(c), k\}$ are smaller than $\min\{v_i(c), k\}$ for each $r < i < s$. But this would
contradict the valuation relation above with $\ell = 1$.

Hence the sequence $\min\{v_i(\mathbf{c}),k\}$ is either non-increasing
(giving~\eqref{APeqn}) or non-decreasing (giving~\eqref{APeqn_increasing}).
Note that Lemma~\ref{lem:relation}$(2)$ implies that $nD_1(\mathbf{c})$ is
in the ideal generated by $(D_2(\mathbf{c}), \ldots, D_n(\mathbf{c}))$, which restricts
$b \leq \min\{\mathrm{val}_p(n), k\} \leq n$ in~\eqref{APeqn_increasing}.
\end{proof}

Lemma~\ref{lem:arithmetic_progression_disc} implies that the support of
$\widehat{\psi_{p^{2k}}}$ is also on ``near arithmetic progressions.''

\begin{lemma}[Support on near arithmetic progressions]\label{lem:arithmetic_progression}
For $\mathbf{u} = (u_1, \ldots, u_n) \in (\mathbb{Z}/p^{2k}\mathbb{Z})^n$, we have that
\begin{equation*}
\widehat{\psi_{p^{2k}}}(\mathbf{u}) = 0
\end{equation*}
unless $\mathbf{u}$ satisfies one of the two ``near arithmetic progression'' properties that
\begin{equation}\label{eq:near_arithmetic_progression}
\min\{ v_p(u_i), k \} = \min \{ v_p(u_n) + (n-i) a, k \}
\end{equation}
for some $a \in \mathbb{Z}_{\geq 0}$, or
\begin{equation}\label{eq:near_arithmetic_progression_increasing}
\min\{ v_p(u_i), k \} = \min \{ v_p(u_1) + (i-1) b, k \}
\end{equation}
for some $b \in \mathbb{Z}$ with $0 \leq b \leq \min\{v_p(n), k\}$.
\end{lemma}

\begin{proof}
For each $\pmb{c} \in (\mathbb{Z}/p^{2k}\mathbb{Z})^n$ such that $p^{2k} \mid
\disc(f_{\pmb{c}})$, we associate the locus
\begin{equation*}
P_{\mathbf{c}}
:=
\{ \mathbf{v} \in (\mathbb{Z}/p^{2k}\mathbb{Z})^n
:
p^{2k} \mid \disc(f_{\mathbf{v}}),
\mathbf{v}-\mathbf{c}\in p^k \cdot (\mathbb{Z}/p^{2k}\mathbb{Z})^n
\}.
\end{equation*}
Two such $P_{\mathbf{c}}$ and $P_{\mathbf{c}'}$ are equal if
and only if $\mathbf{c} - \mathbf{c}' \equiv 0 \bmod p^k$.
Thus we can decompose the set of congruence classes
of polynomials whose discriminant is divisible by $p^{2k}$
into a disjoint union of $P_{\mathbf{c}}$ over a set $C$ of representatives $\mathbf{c}$
from $p^k (\mathbb{Z} / p^{2k} \mathbb{Z})^n$ satisfying $p^{2k} \mid
\disc(f_{\mathbf{c}})$, giving
\begin{equation*}
\widehat{\psi_{p^{2k}}}
=
\sum_{\mathbf{c} \in C} \widehat{\mathbf{1}_{P_{\mathbf{c}}}}.
\end{equation*}

This leads us to study $\widehat{\mathbf{1}_{P_{\mathbf{c}}}}$.
Below, we use the notation $\mathbf{D}_{\mathbf{c}}$ for the gradient vector of the
discriminant function and the notation $v_i$ for the $p$-valuation for the
$i$-th coordinate of $\mathbf{D}_{\mathbf{c}}$ as in Lemma~\ref{lem:arithmetic_progression_disc}.
From the Taylor expansion
\begin{equation*}
\disc(f_{\mathbf{v}})
\equiv
\disc(f_{\mathbf{c}})
+
\mathbf{D}_{\mathbf{c}} \cdot (\mathbf{v}-\mathbf{c})\mod p^{2k}
\end{equation*}
and the fact that $p^{2k} \mid \disc(f_{\mathbf{c}})$, we see that $P_{\mathbf{c}}$ can be written as
\begin{equation}\label{otherdefnofpc}
P_{\mathbf{c}}
=
\Big\{
\mathbf{v} \in (\mathbb{Z}/p^{2k}\mathbb{Z})^n
:
\mathbf{D}_{\mathbf{c}} \cdot \frac{\mathbf{v}-\mathbf{c}}{p^k} \equiv 0 \mod p^k
\Big\}.
\end{equation}
Direct computation on the definition shows
\begin{equation}\label{defnofFT}
\widehat{\mathbf{1}_{P_{\textbf{c}}}}(\pmb{\xi})
=
\frac{1}{p^{2kn}}\exp{\left(2\pi i \frac{\langle \pmb{\xi}, \mathbf{c}\rangle}{p^{2k}}\right)}
\sum_{\mathbf{v}\in P_{\mathbf{c}}}
\exp{\Big(
2\pi i \frac{\langle \pmb{\xi}, \mathbf{v}-\mathbf{c}\rangle}{p^{2k}}
\Big)}.
\end{equation}
Let $w := \min (v_i(\mathbf{c}),k)$ denote the minimum valuation among the
coordinates of $\mathbf{D}_{\mathbf{c}}$. We will now show that
\begin{equation}\label{eq:Pc_support}
\widehat{\mathbf{1}_{P_{\textbf{c}}}}(\pmb{\xi})
= \begin{cases}
\exp(2\pi i \frac{\pmb{\xi}\cdot \textbf{c}}{p^{2k}})\cdot \frac{p^{w-k}} {p^{kn}},
& \pmb{\xi}  \equiv  \alpha\textbf{D}_{\textbf{c}} (\text{mod } p^k) \text{ for some } \alpha \in \mathbb{Z}/p^{k}\mathbb{Z}
\\
0 &\text{otherwise}.
\end{cases}
\end{equation}
To see this, first note that by~\eqref{otherdefnofpc}, the set $P_{\textbf{c}}-\textbf{c}$ forms an
additive group. This implies that
\begin{equation}\label{Gausssumtypeidentity}
\sum_{\mathbf{v} \in P_{\mathbf{c}}}
\exp{\Big(
2\pi i \frac{\langle \pmb{\xi}, \mathbf{v}-\mathbf{c}\rangle}{p^{2k}}
\Big)}
=
\exp{\Big(
2\pi i \frac{\langle \pmb{\xi},  \mathbf{u}-\mathbf{c}\rangle}{p^{2k}}
\Big)}
\sum_{\mathbf{v} \in P_{\mathbf{c}}}
\exp{\Big(
2\pi i \frac{\langle \pmb{\xi}, \mathbf{v}-\mathbf{c}\rangle}{p^{2k}}
\Big)}
\end{equation}
for any $\textbf{u} \in P_{\textbf{c}}$.

When $\pmb{\xi} \equiv \alpha \mathbf{D}_{\mathbf{c}} \bmod p^k$ for some
$\alpha \in \mathbb{Z}/p^k \mathbb{Z}$, all summands in~\eqref{defnofFT}
are equal to $1$ by the definition of $P_{\textbf{c}}$
in~\eqref{otherdefnofpc}. The first case of~\eqref{eq:Pc_support} follows
from the fact that $\lvert P_{\mathbf{c}} \rvert = p^{k(n-1)+w}$.
If $\pmb{\xi}$ cannot be written as $\alpha \textbf{D}_{\textbf{c}}$, then
there must be some $\mathbf{u} \in P_{\mathbf{c}}$ such that $\langle \pmb{\xi}, \mathbf{u} - \mathbf{c} \rangle \not \equiv 0 \bmod p^{2k}$.
Inserting this choice of $\mathbf{u}$ into~\eqref{Gausssumtypeidentity} shows
that $\widehat{\mathbf{1}_{P_{\mathbf{c}}}}(\pmb{\xi}) = 0$.
Thus~\eqref{eq:Pc_support} is shown.

It follows that for all $\mathbf{u}$ in the support of
$\widehat{\psi_{p^{2k}}}$, $\mathbf{u} \equiv \alpha \mathbf{D}_\mathbf{c}$
for some $\mathbf{c}$. The near arithmetic progression
conditions~\eqref{eq:near_arithmetic_progression} and~\eqref{eq:near_arithmetic_progression_increasing} are then implied by Lemma~\ref{lem:arithmetic_progression_disc}.
\end{proof}

For the phases $\mathbf{u} = (u_1, \ldots, u_n)$ for which the Fourier transform $\widehat{\psi_{p^{2k}}}(\mathbf{u})$ does not vanish, we will need an improvement over the trivial bound $\lvert \widehat{\psi_{p^{2k}}}(\mathbf{u})\rvert \leq \lvert\widehat{\psi_{p^{2k}}}(\mathbf{0})\rvert \ll_n p^{-k-\frac{2k}{n}}$ given by Lemma~\ref{lem:discriminant-classes}.
In fact, for our purposes, it will suffice to restrict our attention to those phases for which only $u_1$ and $u_2$ are possibly non-zero.
For the phases $\mathbf{u} = (u_1, \ldots, u_n)$ in which only $u_1$ is non-zero, we observe that the Fourier transform is typically $0$.

\begin{lemma} \label{lem:disc-fourier-bound-u1}
Let $p$ be prime, $n \geq 6$, and $k \geq 3$.  Suppose that $u_1 \in \mathbb{Z}$, and define $\mathbf{u}\in \mathbb{Z}^n$ by $\mathbf{u}=(u_1,0,0,\dots,0)$.  Then
$\widehat{\psi_{p^{2k}}}(\mathbf{u}) = 0$ if
$p^{2k} / \gcd(p^{2k}, n)$ does not divide $u_1$.
\end{lemma}

\begin{proof}
 Since only $u_1 \ne 0$, by definition, we have
\[
\widehat{\psi_{p^{2k}}}(\mathbf{u})
= \frac{1}{p^{2nk}} \sum_{\mathbf{c} \in (\mathbb{Z}/p^{2k}\mathbb{Z})^n} \psi_{p^{2k}}(f_\mathbf{c}) \exp\left( \frac{2\pi i c_1 u_1}{p^{2k}} \right).
\]
As the map $x \mapsto x + 1$ induces a bijection from
the support of $\psi_{p^{2k}}$ to itself (that is, on the set of polynomials whose discriminant is divisible by $p^{2k}$),
sending $c_1 = c_1(f)$ to $c_1 + n$ for each $f$, we also have
\[
\widehat{\psi_{p^{2k}}}(\mathbf{u}) =
\frac{1}{p^{2nk}} \sum_{\mathbf{c} \in (\mathbb{Z}/p^{2k}\mathbb{Z})^n} \psi_{p^{2k}}(f_\mathbf{c}) \exp\left( \frac{2\pi i (c_1 +n) u_1}{p^{2k}} \right).
\]
This implies that either $\widehat{\psi_{p^{2k}}}(\mathbf{u}) = 0$
or $p^{2k} \mid nu_1$.
\end{proof}

To study phases $\mathbf{u} = (u_1, \ldots, u_n)$ with both $u_1$ and $u_2$ non-zero, we'll use the following additional result.

\begin{lemma} \label{lem:disc-fourier-bound}
Let $p$ be prime, $n \geq 6$, and $k \geq 3$.  Suppose that $u_1,u_2 \in \mathbb{Z}$, and define $\mathbf{u}\in \mathbb{Z}^n$ by $\mathbf{u}=(u_1,u_2,0,\dots,0)$.  Then
\[
\widehat{\psi_{p^{2k}}}(\mathbf{u})
\ll_n p^{-\frac{2nk}{3}} \mathrm{gcd}(u_2,p^{2k})^{\frac{n}{3}}.
\]
\end{lemma}

\begin{proof}
Let $e_{p^{2k}} \colon \mathbb{Z}_p \to \mathbb{C}$ be the continuous extension of the map defined on integers $x$ by $e_{p^{2k}}(x) := \exp(\frac{2\pi i x}{p^{2k}})$.  By definition, it follows that we may rewrite the Fourier transform as
\[
\widehat{\psi_{p^{2k}}}(\mathbf{u})
    = \int_{\mathcal{P}_n(\mathbb{Z}_p)} \psi_{p^{2k}}(f) e_{p^{2k}}(u_1 c_1(f) + u_2 c_2(f))\, d\nu(f),
\]
where $c_1(f)$ and $c_2(f)$ respectively denote the coefficients of $x^{n-1}$ and $x^{n-2}$ in $f$.
Applying the change of variables from Section \ref{change of variables}, we find this integral is equal to
\begin{equation} \label{eqn:fourier-change}
\sum_{[K_p:\mathbb{Q}_p]=n} \frac{\lvert \mathrm{Disc}(K_p) \rvert_{p}^{1/2}}{\lvert
\mathrm{Aut}(K_p) \rvert } \int_{\mathcal{O}_{K_p}} \lvert \mathrm{Disc}(f_\alpha) \rvert_p^{1/2} \psi_{p^{2k}}(f_\alpha) e_{p^{2k}}(-u_1\sigma_1(\pmb{\lambda}) + u_2\sigma_2(\pmb{\lambda})) d\mu(\alpha),
\end{equation}
where $f_\alpha$ is the characteristic polynomial of $\alpha$ and $\sigma_i(\pmb{\lambda})$ denotes the $i$-th elementary symmetric function in the roots $\pmb{\lambda}$.  Fixing the \'etale algebra $K_p = F_1 \times \cdots \times F_r$ (with $n_i := \deg(F_i)$), we note that $\sigma_1$ is linear as a polynomial in $\pmb{\lambda}$, and is therefore also linear in the choice of coordinates $\mathbf{a} \in \mathbb{Z}_p^n$ from Section \ref{change of variables} and the proof of Proposition \ref{prop:etale-density} (which we use in this proof as well).  Additionally, $\sigma_2$ is a quadratic form in $\pmb{\lambda}$ with Gram matrix
\[
Q = \begin{pmatrix}
0 & 1 & 1 & \dots & 1 \\
1 & 0 & 1 & \dots & 1 \\
\vdots & & \ddots & & \vdots \\
1 & 1 & 1 & \dots & 0
\end{pmatrix},
\]
that is, $\sigma_2(\pmb{\lambda}) = \frac{1}{2} \pmb{\lambda}^T Q \pmb{\lambda}$, where we regard $\pmb{\lambda}$ as a column vector.  As $\pmb{\lambda} = M_{K_p} \mathbf{a}$, we therefore find that $\sigma_2$, as a polynomial in $\mathbf{a}$, has Gram matrix $M_{K^p}^T Q M_{K_p}$.  For convenience in what is to come, we note now that $\det Q = (-1)^{n-1} (n-1)$.
Thus $\lvert \det M_{K^p}^T Q M_{K_p} \rvert_p \ll_n \lvert \mathrm{Disc}(K_p) \rvert_p$,
where $M_{K_p}$ is as in Section \ref{change of variables}.

Similarly, using the formula in Section \ref{change of variables}, if one considers $\sigma_2$ (as a polynomial in $\mathbf{a}$) as a quadratic form modulo $p$,
this computation shows that we have $\lvert \Disc(K_p) \rvert_p \ll p^{-\textnormal{corank}_{\mathbb{F}_p}(\sigma_2)}$, as the corank represents the number of column vectors that are linearly dependent $p$-adically.

For each $\mathbf{a} = (a_1, \ldots, a_n) \in \mathbb{Z}_p^n$ associated to a point $\alpha \in \mathcal{O}_{K_p}$ in the support of $\psi_{p^{2k}}$,
we introduce the following quantities:

\begin{itemize}
\item We set $\ell = k - \frac{v_p(u_2)}{2}$, and
in the following we will write $a \equiv b \bmod p^\ell$ to mean that $v_p(a - b) \geq \ell$, including the case when $\ell$ is a half integer.
\item We write $S$ for the set of $i$ (or, by abuse of notation, the set of $a_i$) with the following property: one of the roots
for which $a_i$ is a coefficient is  subject to a congruence $\lambda \equiv \lambda^\prime \pmod{p^{\ell}}$, where $\lambda^\prime$ is some other root, and where the congruence is as polynomials in the appropriate integral bases.
\item We write $m := \lvert S \rvert$.
\end{itemize}

If $a_i \in S$,
then using the Galois action we see that \emph{every} root for which $a_i$ is a coefficient will be subject to such a congruence, and thus we may regard this congruence as being of the form $\alpha \equiv \gamma(\alpha') \pmod{p^\ell}$ for some nontrivial $\gamma \in \mathrm{Gal}(\overline{\mathbb{Q}_p}/\mathbb{Q}_p)$.

We distinguish two cases:
\begin{enumerate}
    \item If $\alpha_j \equiv \gamma(\alpha_j) \pmod{p^\ell}$ for some nontrivial $\gamma\in \mathrm{Gal}(\widetilde{F_j}/\mathbb{Q}_p)$,
    let $\lvert \gamma \rvert$ denote the order of $\gamma$.
    Taking traces shows that $\lvert \gamma \rvert\alpha_j$ must be congruent $\pmod{p^\ell}$ to an element of the subfield $F_j^\gamma$ of $F_j$ fixed by $\gamma$.
    Since $\gamma$ acts nontrivially on the conjugates of $F_j$, this subfield has degree at most $n_j / 2$, from which it follows that at least $n_j/2$ values $r_j$ are greater than $\ell$.
    This will be satisfied for a proportion of $O(p^{-\ell n_j/2})$ of the possible $\alpha_j$.
    \item We divide the set of remaining $\alpha$ into equivalence classes where
    $\alpha_{j_1} \equiv \gamma_2(\alpha_{j_2}) \equiv \cdots \equiv \gamma_t(\alpha_{j_t}) \pmod{p^\ell}$ for some $t \geq 2$.  Once the $\gamma$ are fixed, any one of the $\alpha_j$ determines all the rest
    $\pmod{p^\ell}$. Assuming without loss of generality that $\alpha_{j_1}$ is of minimal degree,
    these congruences will be satisfied for a proportion of $O(\min(p^{-\ell(n_{j_2} + \dots + n_{j_t})}))$ of the possible tuples
    $(\alpha_{j_1}, \dots, \alpha_{j_t})$.
\end{enumerate}
There are $O_n(1)$ ways of partitioning the set of $\alpha$ into equivalence classes as above and choosing the $\gamma$.
Altogether we obtain a density $\ll_n p^{-\ell N/2}$, where $N$ is the total of the degrees $n_j$ over those $\alpha_j$ corresponding to at least one coordinate in $S$.
As $N \geq m = \lvert S \rvert$, this density is $\ll_n p^{-\ell m/2}$. Moreover, the above procedure identifies $\geq m/2$ pairs of roots, distinct but not necessarily
disjoint, which are congruent $\pmod{p^\ell}$, so we have
\begin{equation}\label{eq:bound_dfa}
    \lvert \mathrm{Disc}(f_\alpha) \rvert_p^{1/2} \leq p^{-\frac{m\ell}{2}} \lvert
    \mathrm{Disc}(K_p) \rvert_p^{1/2}
\end{equation}
for any such $\alpha$. So far we obtain a bound
\begin{equation}\label{eq:bound_part1}
\ll p^{-m\ell} \lvert \mathrm{Disc}(K_p) \rvert_p^{1/2}
\end{equation}
for the integral in \eqref{eqn:fourier-change}.

For the $n-m$ remaining coordinates $a_i \not \in S$, we may change $a_i$ by an arbitrary multiple of $p^{\lceil\ell\rceil}$ while preserving the discriminant.  Motivated by the method of (non)stationary phase, by Taylor expanding, we then observe that
\[
u_2 \sigma_2(\mathbf{a}+b_ip^{\lceil\ell\rceil}) - u_1 \sigma_1(\mathbf{a} + b_ip^{\lceil\ell\rceil})
\equiv u_2 \sigma_2(\mathbf{a}) - u_1 \sigma_1(\mathbf{a}) + b_i p^{\lceil\ell\rceil}\cdot \Big(u_2 \frac{\partial \sigma_2}{\partial a_i}(\mathbf{a}) - u_1 \frac{\partial \sigma_1}{\partial a_i}(\mathbf{a})\Big)\!\!\!\! \pmod{p^{2k}},
\]
where we use $b_i$ to indicate a vector whose only nonzero component is $b_i$ in the $i$-th coordinate.
We have used that $\sigma_1(\mathbf{a})$ is linear in $a_i$ and the relationship between $\ell$ and $u_2$ in omitting higher derivatives.
This implies that the integral
 \[
    \int_{b_i \in \mathbb{Z}_p} \lvert \mathrm{Disc}(f_\alpha) \rvert_p^{1/2} \psi_{p^{2k}}(f_\alpha) e_{p^{2k}}(-u_1\sigma_1(\mathbf{a}+b_ip^{\lceil\ell\rceil}) + u_2\sigma_2(\mathbf{a}+b_ip^{\lceil\ell\rceil})) \,d b_i
 \]
will vanish unless the congruence $u_2 \frac{\partial \sigma_2}{\partial a_i}(\mathbf{a}) - u_1 \frac{\partial \sigma_1}{\partial a_i}(\mathbf{a}) \equiv 0 \pmod{p^{2k-\lceil\ell\rceil}}$ is satisfied.  Consequently, we may restrict the integral
\begin{equation}\label{eq:int_res}
\int_{\mathcal{O}_{K_p}} \lvert \mathrm{Disc}(f_\alpha) \rvert_p^{1/2} \psi_{p^{2k}}(f_\alpha) e_{p^{2k}}(-u_1\sigma_1(\pmb{\lambda}) + u_2\sigma_2(\pmb{\lambda})) d\mu(\alpha)
\end{equation}
to those points satisfying the congruence $u_2 \frac{\partial \sigma_2}{\partial a_i}(\mathbf{a}) -
u_1 \frac{\partial \sigma_1}{\partial a_i}(\mathbf{a}) \equiv 0 \pmod{p^{2k-\lceil\ell\rceil}}$ for
each $a_i \not \in S$.  Since $\sigma_2$ is quadratic, these are linear constraints, and they will
be independent over $\mathbb{Z}_p$ since the coefficients of $\frac{\partial \sigma_2}{\partial
a_i}$ are precisely the entries of the $i$-th column of the Gram matrix $M_{K_p}^T Q M_{K_p}$, which
is nonsingular by our earlier computation.  In fact, the same computation reveals that the density
of points $\mathbf{a} \in \mathbb{Z}_p^n$ satisfying these conditions will be at most
$p^{-(n-m)\lfloor\ell\rfloor} \lvert \mathrm{Disc}(K_p) \rvert_p^{-1}$.  This bound, multiplied by
\eqref{eq:bound_dfa}, is $O(p^{-(n - \frac{m}{2})\ell} \lvert \Disc(K_p) \rvert^{-1/2})$.

To obtain the same bound if $\ell$ is a half-integer (so that $\lfloor \ell\rfloor = \ell -\frac{1}{2}$),
we begin by recasting the previous argument slightly. We have seen that, if $\alpha$ meets the conditions described by $S$, then
the integrand in \eqref{eq:int_res} is constant in the box $\mathbf{a} + p^{\lceil \ell \rceil} \Z_p^{n - m}$, and that the total measure
of these boxes is $p^{-(n-m)\lceil \ell \rceil} \lvert \mathrm{Disc}(K_p) \rvert_p^{-1}$.

We now slightly enlarge the boxes, dividing $\Z_p^{n - m}$ into boxes of the form $\mathbf{a} + p^{\lfloor \ell \rfloor} \Z_p^{n - m}$. Here
a choice of the coordinates in $S$ will be fixed, and each $\mathbf{a}$ will be chosen to meet the
conditions described by $S$ and satisfy $\psi_{p^{2k}}(f_{\mathfrak{a}}) = 1$. The following will hold for all $\mathbf{a}' = \mathbf{a} + p^{\lfloor \ell \rfloor} \mathbf{b}  \in
\mathbf{a} + p^{\lfloor \ell \rfloor} \Z_p^{n - m}$:
\begin{enumerate}
    \item We will have $\lvert \Disc(f_{\mathbf{a}'}) \rvert_p \leq \lvert \Disc(f_{\mathbf{a}})
    \rvert_p$, with strict inequality if and only if $\mathbf{b}$ satisfies any of a nonempty finite set of $\Z_p$-linear constraints $\pmod{p}$.
    \item $\psi_{p^{2k}}(f_\mathbf{a})$ will be identically $1$.
    \item The quantity $-u_1 \sigma_1(\pmb{\lambda}) + u_2 \sigma_2(\pmb{\lambda}) \pmod{p^{2k}} $ will be constant on boxes of side length $p^{\lfloor \ell \rfloor}$, and hence may be written as the sum of a constant depending only on $\mathbf{a}$, and a quadratic polynomial $q(\mathbf{b})$ depending only on $\mathbf{b} \pmod p$.
    \item After an invertible $\Z_p$-linear change of variables, we may write $q(\mathbf{b}) = c_1 b_1^2 + \cdots + c_t b_t^2$ $+$ (terms involving only $b_{t + 1}, \dots, b_{n - m}$), where we have $n - m - t \leq \textnormal{corank}_{\mathbb{F}_p}(\sigma_2)$.
\end{enumerate}

Ignoring the condition (1) for now, the integral over $\mathbf{a} + p^{\lfloor \ell \rfloor} \mathbf{b}$ reduces to
$\lvert\Disc(f_{\mathbf{a}})\rvert^{1/2}_p p^{-(n - m) \lceil \ell \rceil}$ times a sum over $\mathbf{b} \pmod{p} \in \mathbb{F}_p^{n -m}$, which factors as a product
of $t$ Gauss sums, bounded above by $p^{t/2}$, and a sum of the other $n - m - t$ variables, which is
$\leq p^{n - m - t} \leq p^{\textnormal{corank}_{\mathbb{F}_p}(\sigma_2)} \ll \lvert \Disc(K_p)\rvert_p^{-1}$.
Altogether we obtain a bound
\begin{align}\label{eq:bound1}
\ll & \ \lvert \Disc(f_{\mathbf{a}})\rvert^{1/2}_p p^{-(n - m) \lceil \ell \rceil} p^{t/2} \lvert \Disc(K_p)\rvert^{-1} \nonumber \\
\ll & \ \lvert \Disc(f_{\mathbf{a}})\rvert^{1/2}_p p^{-(n - m) \ell} \lvert \Disc(K_p)\rvert^{-1} \nonumber \\
\ll & \ p^{-(n - \frac{m}{2})\ell} \lvert\Disc(K_p)\rvert^{-1/2}.
\end{align}
We handle the condition (1) by an inductive argument. Passing to the set of $\mathbf{b}$ satisfying any one of the linear constraints, and
choosing a $\Z_p$-basis for this set, (1)-(4) will remain true, with the rank $t$ of the Gauss sum either remaining the same, or decreasing
by $1$. Since $\lvert \Disc(f) \rvert_p$ decreases by a factor of $p$, as does the measure of the set being integrated over, while $t$ may decrease
by $1$, we obtain a bound a factor of $O(p^{-1})$ smaller than \eqref{eq:bound1}. Continuing to pass to such subsets, and using
inclusion-exclusion as needed, we obtain the same bound $\eqref{eq:bound1}$ incorporating (1).

Our density bounds for $a_i \in S$ and $a_i \not \in S$ have not been proved independent,
so adding the minimum of \eqref{eq:bound_part1} and \eqref{eq:bound1} over the $2^n$ choices for $S$, and
noting
that $\lvert \Disc(K_p) \rvert_p \leq 1$,
we obtain
\[
        \frac{\lvert \mathrm{Disc}(K_p) \rvert_{p}^{1/2}}{\lvert \mathrm{Aut}(K_p) \rvert }
        \int_{\mathcal{O}_{K_p}} \lvert \mathrm{Disc}(f_\alpha) \rvert_p^{1/2} \psi_{p^{2k}}(f_\alpha) e_{p^{2k}}(-u_1\sigma_1(\pmb{\lambda}) + u_2\sigma_2(\pmb{\lambda})) d\mu(\alpha)
            \ll_n p^{-\frac{2n\ell}{3}}.
    \]
Adding this across the $O_n(1)$ \'etale algebras, and recalling that $\ell = k - \frac{v_p(u_2)}{2}$, we deduce the result.
\end{proof}

\section{Polynomials with a large and powerful index}\label{sec:large-radical}

The main goal of this section is to prove Theorem~\ref{thm:small-radical-index}, which states that the number of polynomials for which  $\idx(f) \gg_n H^{\frac{n(n-3)}{2}}$ but $\mathrm{rad}(\idx(f)) < H^{1 - \epsilon}$ is $O_{n,\epsilon}(H^{\frac{n^2+n}{2}-\frac{4}{3}-\frac{4}{n} + \epsilon}+H^{\frac{n^2+n}{2}-\frac{2n}{3}+3+\epsilon})$.

We begin with the following elementary lemma that, in the context of Theorem \ref{thm:small-radical-index}, will allow us to choose a convenient divisor of the index.

\begin{lemma} \label{lem:powerful-divisor}
Let $m \geq 2$ be an integer, and let $C = \rad(m)$ be the product of the primes dividing $m$.  Let $k \geq 2$.  If $m \geq C^{2k-2}$, then for every $x \in \mathbb{R}$ such that $C^{k-1} \leq x \leq m/C^{k-1}$,  $m$ has a $k$-powerful divisor $d$ in the interval $[x,Cx]$.
\end{lemma}

(Recall that an integer $d$ is $k$-powerful if every prime $p \mid d$ divides $d$ to order at least $k$.)

\begin{proof}
We begin by proving a slightly stronger statement in the case that $m$ is itself $k$-powerful, in particular producing a $k$-powerful divisor in every interval of the form $[x,Cx]$ with $C^{k-1} \leq x \leq m/C$.  If $x \leq C^k$, then we simply take the divisor $d$ to be $C^k$.  If $x>C^k$, then we consider divisors of the form $C^k a$ with $a$ a divisor of $m/C^k$, and we claim there must be such a divisor $a$ in the interval $[x/C^k, x/C^{k-1}]$.  If this interval includes $m/C^k$, then we take the divisor $a$ to be $m/C^k$ itself.  Otherwise, let $a$ be the minimal divisor of $m/C^k$ greater than $x/C^{k-1}$.  Then every prime divisor $p$ of $a$ is at most $C$, and thus $a/p \geq x/C^k$.  On the other hand, by the minimality assumption on $a$, $a/p \leq x/C^{k-1}$, and thus $a/p$ is the claimed divisor.  This completes the proof in the case $m$ is $k$-powerful.

If $m$ is not $k$-powerful, let $m^\prime$ denote the maximal $k$-powerful divisor of $m$ and let $C^\prime = \rad(m^\prime)$, and notice that $C^\prime \leq C$.  Additionally, we have $m/C^{k-1} \leq m^\prime / {C^\prime}^{(k-1)}$.  It then follows from the previous paragraph that $m^\prime$ has a $k$-powerful divisor in the interval $[x,C^\prime x] \subseteq [x,Cx]$ for $x$ as in the statement of the lemma, and thus $m$ must too.
\end{proof}

We now turn to the proof of Theorem \ref{thm:small-radical-index}.

\begin{proof}[Proof of Theorem \ref{thm:small-radical-index}]
We wish to show that the number of monic polynomials $f(x) \in \mathbb{Z}[x]$ of degree $n$ and height $H$ for which $\mathrm{rad}(\mathrm{index}(f)) < H^{1-\epsilon}$ but $\mathrm{index}(f) > H^{\frac{n(n-3)}{2}}$ is $O_{n,\epsilon}(H^{\frac{n^2+n}{2}-\frac{4}{3}-\frac{4}{n} + \epsilon}+H^{\frac{n^2+n}{2}-\frac{2n}{3}+3+\epsilon})$.
Suppose we are considering such a polynomial $f$.  By Lemma \ref{lem:powerful-divisor} applied with $k=3$, there is some cubefull divisor $d$ of the index of $f$ satisfying $H^{2-2\epsilon} < d \leq H^{3-3\epsilon}$.  By the remainder theorem, given such a divisor $d$, the polynomial $f$ will lie in the support of the function $\psi_{d^2} = \prod_{p^k \mid\mid d} \psi_{p^{2k}}$.  Moreover, the support of this function is determined by congruence conditions $\pmod{d^2}$.  Since $d^2 > H$, we may not simply estimate the number of such polynomials by na\"ively counting the number in each residue class.   Instead, we identify the space of monic, degree $n$ polynomials with $\mathbb{Z}^n$ and we let $\phi\colon \mathbb{R}^n \to \mathbb{R}$ be a non-negative Schwartz function, chosen so that it is greater than $1$ on the unit box $[-1,1]^n$ and so that its Fourier transform $\widehat\phi$ has compact support contained in $[-\frac{1}{2n^5},\frac{1}{2n^5}]^n$.

Then by Poisson summation,
\begin{align} \label{eqn:index-poisson}
\sum_{\mathbf{c} \in \mathbb{Z}^n} \phi\Big(\frac{c_1}{H},\ldots,\frac{c_n}{H^n}\Big) \psi_{d^2}(f_\mathbf{c})
&= H^{\frac{n^2+n}{2}}\sum_{\mathbf{u} \in \mathbb{Z}^n} \widehat\phi\Big(\frac{u_1 H}{d^2}, \ldots, \frac{u_n H^n}{d^2}\Big) \widehat\psi_{d^2}(\mathbf{u}) \notag \\
&= H^{\frac{n^2+n}{2}}\widehat\psi_{d^2}(\mathbf{0}) \widehat{\phi}(\mathbf{0}) + H^{\frac{n^2+n}{2}} \sum_{\mathbf{0} \ne \mathbf{u} \in \mathbb{Z}^n} \widehat\phi\Big(\frac{u_1 H}{d^2}, \ldots, \frac{u_n H^n}{d^2}\Big) \widehat\psi_{d^2}(\mathbf{u})
\end{align}
where $\widehat\psi_{d^2}(\mathbf{u})$ is defined by
\[
\widehat\psi_{d^2}(\mathbf{u})
= \frac{1}{d^{2n}} \sum_{f \in (\mathbb{Z}/d^2\mathbb{Z})^{n}} \psi_d(f) e^{\frac{2\pi i \langle f,\mathbf{u}\rangle}{d^2}}.
\]
By the remainder theorem, $\widehat\psi_{d^2}(\mathbf{u})$ may be decomposed as $\widehat\psi_{d^2}(\mathbf{u}) = \prod_{p^k \mid\mid d} \widehat\psi_{p^{2k}}(\gamma_p \mathbf{u})$ for some $\gamma_p$ coprime to $p$.  Thus, $\widehat\psi_{d^2}(\mathbf{u})$ may be bounded by the results of the previous section, in particular Lemma~\ref{lem:discriminant-classes}, Lemma~\ref{lem:arithmetic_progression}, and Lemma~\ref{lem:disc-fourier-bound}.
Using Lemma~\ref{lem:discriminant-classes}, the first term on the right-hand side of \eqref{eqn:index-poisson} is $O_{n,\phi,\epsilon}(H^{\frac{n^2+n}{2}} d^{-1-\frac{2}{n}+\epsilon})$.  Added over cubefull integers $d>H^{2-2\epsilon}$, this yields a total contribution that is $O_{n,\phi,\epsilon}(H^{\frac{n^2+n}{2}-\frac{4}{3}-\frac{4}{n}+\epsilon})$, which matches the bound claimed in Theorem \ref{thm:small-radical-index}.

For the contribution from the sum of the non-trivial Fourier coefficients, we first note that by the
compact support of $\widehat\phi$, the summation is supported on those $\mathbf{u}$ for which each
$\lvert u_i \rvert\leq d^2/2H^i$.  Since we have  assumed that $d \leq H^{3-3\epsilon}$, this implies that $u_i=0$ for $i \geq 6$.
In particular, since $u_6=0$, it follows from the arithmetic progression property of the support (Lemma~
\ref{lem:arithmetic_progression}) of $\widehat\psi_{d^2}$ that each $u_i$ for $i \leq 5$ must be divisible by $d/\gcd(n^{6-i}, d)$. (In the worst case, when~\eqref{eq:near_arithmetic_progression_increasing} holds for every prime $p$ dividing $d$, we bound $b \leq v_p(n)$).
Thus, writing $u_i = (d/\gcd(n^{6-i}, d)) u_i'$, we find that $u_i^\prime$ satisfies $\lvert u_i'
\rvert \leq d / H^{i}$, which in particular also implies that $u_i = 0$ for $i \geq 3$.  If $u_2=0$,
then the sum is only over integers $\lvert u_1 \rvert \leq d^2/H$.  However, by Lemma \ref{lem:disc-fourier-bound-u1}, $\widehat{\psi_{d^2}}(\mathbf{u})$ will vanish unless $u_1$ is divisible by $d^2/\mathrm{gcd}(d^2,n)$,
which forces $u_1=0$ in this range as well.
As we have already considered the contribution from the trivial Fourier coefficient, we may therefore assume that $u_2 \ne 0$.

If $u_2 \ne 0$, then by Lemma \ref{lem:disc-fourier-bound}, $\hat\psi_{d^2}(\mathbf{u}) \ll_{n,\epsilon} d^{-2n/3+\epsilon} \mathrm{gcd}(d^2,u_2)^{n/3} \ll_n d^{-n/3+\epsilon} \mathrm{gcd}(d,u_2')^{n/3}$.  This case thus yields a contribution that is
\begin{equation}\label{eqn:u2}
\ll_{n,\epsilon} \frac{H^{\frac{n^2+n}{2}+\epsilon}}{d^{n/3}} \sum_{\lvert u_1^\prime \rvert \ll_n
d/H} \sum_{0 \ne \lvert u_2' \rvert \ll_n d/H^2} \mathrm{gcd}(d,u_2')^{n/3}
\ll_{n,\epsilon} d H^{\frac{n^2+n}{2}-\frac{2n}{3}-1+\epsilon},
\end{equation}
by noting that the summation over $u_2^\prime$ is dominated by the $O(d^\epsilon)$ values $u_2^\prime$ for which $\mathrm{gcd}(d,u_2^\prime)$ is largest possible, but that this gcd is $O_n(d/H^2)$ from the bound on $u_2^\prime$.
After summing over cubefull integers $H^{2 - 2\epsilon} < d < H^{3 - 3\epsilon}$, we find a total contribution from the non-trivial Fourier coefficients that is $O_{n,\epsilon}(H^{\frac{n^2+n}{2}-\frac{2n}{3}+3+\epsilon})$, which completes the proof of the theorem.
\end{proof}

\section{Strong and weak multiples} \label{sec:strong+weak}

In this section, we prove Theorem~\ref{thm:large-index}, which bounds the number of polynomials of height $H$ whose discriminants have large squarefree divisors.
To do this, we build on the work of Bhargava, Shankar, and Wang \cite{BSW}.
As before, to any $\mathbf{c} = (c_1, \ldots, c_n) \in \mathbb{Z}^n$ we associate the monic degree $n$ polynomial $f_{\mathbf{c}}(x) = x^n + c_1 x^{n-1} + \cdots + c_n$.

Theorem~\ref{thm:large-index} should be compared with \cite[Theorem 4.4]{BSW}, where a similar statement is proven but where the second term is $H^{\frac{n^2+n}{2}-\frac{1}{5}+\epsilon}$ instead of $H^{\frac{n^2+n}{2}-\frac{1}{2}+\epsilon}$.  It should also be compared with \cite[Theorem~4]{BSW2}, where a strictly stronger statement is proven.  Additionally, as sketched in \cite{BSW}, it is possible to prove that the first term on the right-hand side is $H^{\frac{n^2+n}{2}}/M$, but we do not pursue this as it is not necessary in the proof of our main theorem.

To prove Theorem \ref{thm:large-index}, we follow the same strategy as \cite{BSW}, offering improvements to each of the two key steps.  For a squarefree integer $m$, let $\mathcal{W}_m$ be the underlying set of polynomials whose discriminant is divisible by $m^2$, that is
\[
\mathcal{W}_m := \{ {\bf c} \in \Z^n : \lvert c_i \rvert \leq H^i, m^2 \mid \mathrm{disc}(f_{{\bf c}})\}.
\]
We decompose this set as follows.  For a prime $p$, we say that $\mathrm{disc}(f)$ is a \emph{strong multiple} of $p^2$ if $\disc(g)$ is a multiple of $p^2$ for every polynomial $g$ congruent to $f$ $\pmod{p}$, and we say that $\disc(f)$ is a \emph{weak multiple} of $p^2$ otherwise.  Let $\mathcal{W}_m^{(1)} \subseteq \mathcal{W}_m$ be the subset consisting of polynomials for which $\disc(f)$ is a strong multiple of $p^2$ for every prime $p \mid m$, and let $\mathcal{W}_m^{(2)} \subseteq \mathcal{W}_m$ be the subset of those for which $\disc(f)$ is a weak multiple of $p^2$ for every prime $p \mid m$.  By considering the factorization of $m$, Bhargava, Shankar, and Wang \cite{BSW} make the simple observation that for any $M > 1$,
\[
\bigcup_{\substack{m \geq M \\ m \text{ squarefree}}} \mathcal{W}_m
\subseteq \bigcup_{\substack{m \geq \sqrt{M} \\ m \text{ squarefree}}} \mathcal{W}_m^{(1)}
\cup
\bigcup_{\substack{m \geq \sqrt{M} \\ m \text{ squarefree}}} \mathcal{W}_m^{(2)}.
\]
Thus, to prove Theorem \ref{thm:large-index}, it is sufficient to prove the following pair of
propositions that refine~\cite[Theorem~1.5(a)]{BSW} and~\cite[Theorem~1.5(b)]{BSW}, respectively.

\begin{proposition}\label{prop:strong-multiples}
For any $Y > 1$, $H > 0$, and $\epsilon > 0$,
\begin{equation*}
\bigcup_{\substack{m \geq Y \\ m \; \text{squarefree}}}
\mathcal{W}_m^{(1)}
\ll_{n, \epsilon}
\frac{H^{\frac{n^2+n}{2} + \epsilon}}{Y}
+
H^{\frac{n^2+n}{2} - n + 1 + \epsilon}.
\end{equation*}
\end{proposition}

\begin{proposition}\label{prop:weak-multiples}
For any $Y>1$ and any $\epsilon > 0$, we have
\[
\bigcup_{\substack{m \geq Y \\ m \text{ squarefree}}} \mathcal{W}_m^{(2)}
\ll_n \frac{H^{\frac{n^2+n}{2}}}{Y} + H^{\frac{n^2+n}{2}-\frac{1}{2} + \epsilon}.
\]
\end{proposition}

We now turn to the proofs of these propositions.

\subsection{Strong multiples and Proposition \ref{prop:strong-multiples}}

To prove Proposition~\ref{prop:strong-multiples}, we apply a line of reasoning similar to the geometric sieve (c.f. Theorem~3.3 of~\cite{Bhargava_geosieve}).

Recall that each $\mathbf{c} \in \mathbb{Z}^n$ corresponds to a degree $n$ polynomial
$f_{\mathbf{c}} = x^n + c_1 x^{n-1} + \cdots + c_n$. The discriminant
$\disc(f_{\mathbf{c}})$ is a polynomial in the coefficients $c_i$.
More generally, to any $\mathbf{c} \in \mathbb{R}^n$ we can associate the gradient vector
$\mathbf{D}_{\mathbf{c}} := ( \frac{\partial \disc(f)}{\partial c_1}, \ldots,
\frac{\partial \disc(f)}{\partial c_n})$.
Then we see that $\disc(f_{\mathbf{c}})$ is a strong multiple of $p^2$ if and only if both
$\disc(f_\mathbf{c}) \equiv 0 \bmod p$ and $\mathbf{D}_{\mathbf{c}} \equiv \mathbf{0} \bmod p$.
Thus it suffices to bound the number of $\mathbf{c}$ with $\lvert c_i \rvert \leq H^i$ ($1 \leq
i \leq n$) such that $\disc(f_\mathbf{c}) \equiv 0 \bmod q$ and
$\frac{\partial \disc(f)}{\partial c_n} \equiv 0 \bmod q$ for some squarefree integer $q \geq Y$.
To simplify, we replace $\frac{\partial \disc(f)}{\partial c_n}(\mathbf{c})$ by the resultant
of $\frac{\partial \disc(f)}{\partial c_n}(\mathbf{c})$ and $\disc(f_\mathbf{c})$ with respect
to $c_{n-1}$. Denote the two polynomials $\disc(f_\mathbf{c})$ and
$\mathrm{Res}(\disc(f_\mathbf{c}), \frac{\partial \disc(f)}{\partial c_n})$ by
$g_1(\mathbf{c})$ and $g_2(\mathbf{c})$, respectively, where we notate $g_2(\mathbf{c})$
\textit{even though $g_2$ doesn't depend on $c_{n-1}$} so that it suffices to count common
roots $\mathbf{c}$ of $g_1$ and $g_2$.

We first show that $g_1(\mathbf{c})$, as a polynomial in $(\mathbb{Z}[c_1, \ldots, c_{n-2},
c_n])[c_{n-1}]$, has degree $n$ and nonzero constant leading coefficient, or equivalently that
$g_1(\mathbf{c}) = \sum_{i \leq n} \alpha_i c_{n-1}^i$ with $\alpha_n \neq 0$.
To see this, we compute $g_1(\mathbf{c}) = \disc(f_\mathbf{c})$ via the resultant matrix
$A$ of the monic polynomial $f_{\mathbf{c}}(x)$ and its derivative $f'_{\mathbf{c}}(x)$.
In this matrix, we observe that $c_{n-1}$ occurs exactly $0$ times in the first $n-1$ rows,
twice in rows $n$ to $(2n-2)$, and once in the $(2n-1)$st row.
Thus the degree is at most $n$. We explicitly compute that the coefficient $\alpha_n$ is
\begin{equation*}
\alpha_n = \sum_{0 \leq k \leq n-1} (-n)^k \cdot \binom{n-1}{k}
= (1 - n)^{n-1}
\neq 0.
\end{equation*}

Next, from the weighted homogeneity property of resultants, we observe that $g_2(\mathbf{c})$ may be expressed as a homogeneous polynomial of degree $n^2(n-2)$ in the roots of $f$, and is therefore a weighted homogeneous polynomial in $\mathbf{c}$ of the same degree, where each $c_i$ has weight $i$.  It therefore has degree at most $n(n-2)$ as a polynomial in $c_n$, and we observe that there is a unique permutation giving rise to a term of such degree in the determinant representation of the resultant, corresponding to choosing the top entry of the columns associated to $\mathrm{disc}(f_\mathbf{c})$ and the bottom entry of the columns associated to $\frac{\partial \disc(f)}{\partial c_n}$.  We conclude that $g_2(\mathbf{c})$, as a polynomial in $c_n$, has degree $n(n-2)$ and that its leading coefficient is a non-zero constant depending only on $n$.  In particular, $g_2(\mathbf{c})$ will be a non-zero polynomial in $c_n$ for any choice of the coordinates $(c_1,\dots,c_{n-2})$.

The total number of tuples among the first $n-2$ coordinates $(c_1, \ldots, c_{n-2})$
satisfying $\lvert c_i \rvert \leq H^i$ is of the order $H^{n(n+1)/2 - (2n - 1)}$.
For each such tuple, the number of $c_n$ such that $g_2(c_1, \ldots, c_{n-2}, c_n) = 0$ is
absolutely bounded by the degree in $c_n$, and thus bounded by $O_n(1)$; the number of
$c_{n-1}$ is trivially bounded by $H^{n-1}$.
These contribute at most $O_n(H^{n(n+1)/2 - n})$ many $\mathbf{c}$.

If $g_2(c_1, \ldots, c_{n-2}, c_n) \neq 0$, then we take $q \geq Y$ to be any squarefree
divisor of $g_2(c_1, \ldots, c_{n-2}, c_n)$; there are at most $O_{n,
\epsilon}(H^\epsilon)$ many such choices $q$. The number of $c_{n-1}$ with $\lvert
c_{n-1} \rvert \leq H^{n-1}$ such that $g_1(c_1, \ldots, c_n) \equiv 0 \bmod q$ is at most
$O_n(\max\{1, H^{n-1}/Y\})$. In total, these contribute at most
$O_n(H^{n(n+1)/2 + \epsilon} / Y) + O_n(H^{n(n+1)/2 + \epsilon - (n - 1)})$.

Combining these two bounds completes the proof of Proposition~\ref{prop:strong-multiples}.\qed{}

\subsection{Weak multiples and Proposition \ref{prop:weak-multiples}}

We now turn to Proposition \ref{prop:weak-multiples}, which is an improvement of the
second error term in \cite[Theorem 1.5(b)]{BSW}. Our input consists simply of
the application of a sharper sieve result at the heart of their argument.
In particular, the error term in question
originates in Propositions 2.6 and 3.5 of \cite{BSW}, depending on whether $n$ is odd or even.  The proofs of both propositions rely on an application of the Selberg sieve to bound the number of ``distinguished orbits'' of an appropriate orthogonal group acting on the space of symmetric $n\times n$ matrices over $\frac{1}{2}\mathbb{Z}$. For each prime $p$, they establish
that they are able to sieve out a proportion of congruence conditions which is uniformly
bounded away from $0$, and then an application of the Selberg sieve yields their result.

It turns out that an application of the large sieve in the form
of \cite[Theorem 10.1.1]{Serre} immediately yields the stated stronger result.
We also note that further improvements might be possible, in particular by incorporating
Fourier analysis when applying the Selberg sieve as a large sieve, as done in \cite{AIM1}.

\section*{Acknowledgements}
This collaboration arose initially from an AIM workshop on Fourier Analysis, Arithmetic Statistics, and Discrete Restriction organized by two of the authors (T.A. and F.T.) and Trevor Wooley. The authors thank AIM for a very supportive working environment.
The authors thank Manjul Bhargava, Arul Shankar, and Jerry Wang for sharing a preprint version of their manuscript \cite{BSW2}, and for many useful conversations.
They would also like to thank the anonymous referee for their many useful comments.

\bibliographystyle{alpha}

\begin{thebibliography}{AGLO{\etalchar{+}}23}

\bibitem[AGLO{\etalchar{+}}23]{AIM1}
Theresa~C. Anderson, Ayla Gafni, Robert~J. Lemke~Oliver, David Lowry-Duda,
  George Shakan, and Ruixiang Zhang.
\newblock Quantitative {H}ilbert irreducibility and almost prime values of
  polynomial discriminants.
\newblock {\em Int. Math. Res. Not. IMRN}, (3):2188--2214, 2023.

\bibitem[Bha05]{bhargavaquartic}
Manjul Bhargava.
\newblock The density of discriminants of quartic rings and fields.
\newblock {\em Ann. of Math. (2)}, 162(2):1031--1063, 2005.

\bibitem[Bha10]{bhargavaquintic}
Manjul Bhargava.
\newblock The density of discriminants of quintic rings and fields.
\newblock {\em Ann. of Math. (2)}, 172(3):1559--1591, 2010.

\bibitem[Bha14]{Bhargava_geosieve}
Manjul Bhargava.
\newblock The geometric sieve and the density of squarefree values of invariant
  polynomials, 2014.
\newblock arXiv:1402.0031.

\bibitem[BSW22a]{BSW2}
Manjul Bhargava, Arul Shankar, and Xiaoheng Wang.
\newblock An improvement on {S}chmidt's bound on the number of number fields of
  bounded discriminant and small degree.
\newblock {\em Forum Math. Sigma}, 10:Paper No. e86, 13, 2022.

\bibitem[BSW22b]{BSW}
Manjul Bhargava, Arul Shankar, and Xiaoheng Wang.
\newblock Squarefree values of polynomial discriminants {I}.
\newblock {\em Invent. Math.}, 228(3):1037--1073, 2022.

\bibitem[CDyDO02]{CDDO02}
Henri Cohen, Francisco Diaz~y Diaz, and Michel Olivier.
\newblock Enumerating quartic dihedral extensions of {$\mathbb{Q}$}.
\newblock {\em Compositio Math.}, 133(1):65--93, 2002.

\bibitem[Cou20]{Couveignes}
Jean-Marc Couveignes.
\newblock Enumerating number fields.
\newblock {\em Ann. of Math. (2)}, 192(2):487--497, 2020.

\bibitem[Dav51]{Davenport}
H.~Davenport.
\newblock On a principle of {L}ipschitz.
\newblock {\em J. London Math. Soc.}, 26:179--183, 1951.

\bibitem[DH69]{DH69}
H.~Davenport and H.~Heilbronn.
\newblock On the density of discriminants of cubic fields.
\newblock {\em Bull. London Math. Soc.}, 1:345--348, 1969.

\bibitem[EV06]{EllenbergVenkatesh}
Jordan~S. Ellenberg and Akshay Venkatesh.
\newblock The number of extensions of a number field with fixed degree and
  bounded discriminant.
\newblock {\em Ann. of Math. (2)}, 163(2):723--741, 2006.

\bibitem[GKZ08]{MR2394437}
I.~M. Gelfand, M.~M. Kapranov, and A.~V. Zelevinsky.
\newblock {\em Discriminants, resultants and multidimensional determinants}.
\newblock Modern Birkh\"{a}user Classics. Birkh\"{a}user Boston, Inc., Boston,
  MA, 2008.
\newblock Reprint of the 1994 edition.

\bibitem[LOT22]{LOT20}
Robert~J. Lemke~Oliver and Frank Thorne.
\newblock Upper bounds on number fields of given degree and bounded
  discriminant.
\newblock {\em Duke Math. J.}, 171(15):3077--3087, 2022.

\bibitem[Sch95]{Schmidt}
Wolfgang~M. Schmidt.
\newblock Number fields of given degree and bounded discriminant.
\newblock {\em Ast\'{e}risque}, 4(228):189--195, 1995.
\newblock Columbia University Number Theory Seminar (New York, 1992).

\bibitem[Ser08]{Serre}
Jean-Pierre Serre.
\newblock {\em Topics in {G}alois theory}, volume~1 of {\em Research Notes in
  Mathematics}.
\newblock A K Peters, Ltd., Wellesley, MA, second edition, 2008.
\newblock With notes by Henri Darmon.

\bibitem[ST23]{ShankarTsimerman-Heuristics}
Arul Shankar and Jacob Tsimerman.
\newblock Heuristics for the asymptotics of the number of {$S_n$}-number
  fields.
\newblock {\em J. Lond. Math. Soc. (2)}, 107(5):1613--1634, 2023.

\end{thebibliography}
\newcommand{\etalchar}[1]{$^{#1}$}

\begin{dajauthors}
\begin{authorinfo}[Anderson]
  Theresa C. Anderson\\
  Gregg Zeitlin Assistant Professor of Mathematical Sciences \\
  Carnegie Mellon University\\
  Pittsburgh, PA\\
  tanders2\imageat{}andrew\imagedot{}cmu\imagedot{}edu \\
  \url{https://faculty.math.cmu.edu/anderson/}
\end{authorinfo}
\begin{authorinfo}[Gafni]
  Ayla Gafni\\
  Assistant Professor\\
  University of Mississippi\\
  Oxford, MS\\
  argafni\imageat{}olemiss\imagedot{}edu \\
  \url{https://math.olemiss.edu/ayla-gafni/}
\end{authorinfo}
\begin{authorinfo}[Hughes]
  Kevin Hughes\\
  Lecturer\\
  Edinburgh Napier University\\
  Edinburgh, Scotland\\
  khughes\imagedot{}math\imageat{}gmail\imagedot{}com \\
  \url{https://sites.google.com/site/khughesmath/}
\end{authorinfo}
\begin{authorinfo}[Lemke Oliver]
  Robert J. Lemke Oliver\\
  Associate Professor\\
  Tufts University\\
  Medford, MA\\
  robert\imagedot{}lemke\_{ }oliver\imageat{}tufts\imagedot{}edu \\
  \url{https://rlemke01.math.tufts.edu/}
\end{authorinfo}
\begin{authorinfo}[Lowry-Duda]
  David Lowry-Duda\\
  Senior Research Scientist\\
  ICERM\\
  Providence, RI\\
  david\imageat{}lowryduda\imagedot{}com \\
  \url{https://davidlowryduda.com/}
\end{authorinfo}
\begin{authorinfo}[Thorne]
  Frank Thorne\\
  Professor\\
  University of South Carolina\\
  Columbia, SC\\
  thorne\imageat{}math\imagedot{}sc\imagedot{}edu \\
  \url{https://thornef.github.io/}
\end{authorinfo}
\begin{authorinfo}[Wang]
  Jiuya Wang\\
  Assistant Professor\\
  University of Georgia\\
  Athens, GA\\
  jiuya\imagedot{}wang\imageat{}uga\imagedot{}edu \\
  \url{https://wangjiuya.github.io/}
\end{authorinfo}
\begin{authorinfo}[Zhang]
  Ruixiang Zhang\\
  Assistant Professor\\
  University of California, Berkeley\\
  Berkeley, CA\\
  ruixiang\imageat{}berkeley\imagedot{}edu \\
  \url{https://sites.google.com/view/ruixiang-zhang/home}
\end{authorinfo}
\end{dajauthors}

\end{document}